\documentclass[%
 reprint,
 amsmath,amssymb,
 aps,
]{revtex4-2}

\newcommand\freefootnote[1]{%
  \let\thefootnote\relax%
  \footnotetext{#1}%
  \let\thefootnote\svthefootnote%
}

\usepackage{xcolor}
\usepackage{tikz}
\usetikzlibrary{positioning, arrows.meta}
\usepackage{amssymb}
\usepackage{amsthm}
\newtheorem{theorem}{Theorem}
\newtheorem{proposition}{Proposition}
\newtheorem{corollary}{Corollary}
\usepackage{hyperref}
\usepackage{mathrsfs}

\usepackage[normalem]{ulem}
\usepackage{graphicx}%
\usepackage{dcolumn}%
\usepackage{bm}%

\usepackage[pass]{geometry}

\DeclareMathOperator*{\argmin}{arg\,min}

\usepackage{xr}
\usepackage{algpseudocode}
\usepackage{titlesec}
\usepackage[ruled,vlined]{algorithm2e}
\begin{document}

\preprint{APS/123-QED}

\title{Invariant Measures in Time-Delay Coordinates for\\ Unique Dynamical System Identification}%
\author{Jonah Botvinick-Greenhouse$^{1}$, Robert Martin$^2$, and Yunan Yang$^3$}
 \affiliation{$^1$Center for Applied Mathematics, Cornell University, Ithaca, NY 14853\\
 $^2$DEVCOM ARL Army Research Office, Research Triangle Park, Durham, NC 27709\\
$^3$Department of Mathematics,
Cornell University,
Ithaca, NY 14853}%

\date{\today}%

\begin{abstract}

While invariant measures are widely employed to analyze physical systems when a direct study of pointwise trajectories is intractable, e.g., due to chaos or noise, they cannot uniquely identify the underlying dynamics. Our first result shows that, in contrast to invariant measures in state coordinates, e.g., $[x(t), y(t), z(t)]$, the invariant measure expressed in time-delay coordinates, e.g., $[x(t), x(t-\tau),\ldots, x(t-(m-1)\tau)]$, can identify the dynamics up to a topological conjugacy. Our second result resolves the remaining ambiguity: by combining invariant measures constructed from multiple delay frames with distinct observables, the system is uniquely identifiable, provided that a suitable initial condition is satisfied. These guarantees require informative observables and appropriate delay parameters ($m,\tau$), which can be limiting in certain settings. We support our theoretical contributions through a series of physical examples demonstrating how invariant measures expressed in delay-coordinates can be used to perform robust system identification in practice.

\end{abstract}

\maketitle

\textit{Introduction.}\---- 
Dynamical system parameter identification is a fundamental task across numerous physical applications, e.g., fluid flow surrogate modeling, ion thruster model calibration, gravitational wave parameter estimation and weather prediction. It is typically formulated as an optimization procedure, where observed time-series data is compared to outputs of candidate models to calibrate the best parameters. Sparse sampling, noisy measurements, partial observations, and chaotic dynamics represent common challenges for discovering physically informative models. In particular, directly comparing long time-series data from chaotic systems is known to be impractical, as the intrinsic unpredictability of chaos makes it difficult to differentiate between model inaccuracies and inherent instability~\cite{MOECKEL1997187, yang2023optimal, jiang2023training}.

In contrast, the \textbf{invariant measures} of chaotic systems, which are independent of initial conditions, resilient to measurement noise, and can be accurately approximated even under slow sampling conditions, offer a robust alternative for performing system identification. In several works, invariant measures have been used to quantify the discrepancy between dynamical systems~\cite{greve2019data,chen2023detecting,mezic2004comparison,MUSKULUS201145}, while steady states of the Fokker–Planck equation have been used for fast invariant measure evaluation~\cite{yang2023optimal,botvinick2023learning}. Invariant measures have also been incorporated in regularization terms combined with trajectory-based methods~\cite{jiang2023training,schiff2024dyslim,platt2023constraining,li2022learning,tang2024learning}. Other studies have explored optimal perturbations to produce desired linear responses in invariant measures with applications to climate systems~\cite{kloeckner2018linear,galatolo2017controlling,antown2022optimal,npg-24-393-2017,hairer2010simple}. Novel neural network architectures and loss functions have also been designed to reconstruct invariant measures in dynamical system modeling~\cite{linot2023stabilized,guan2024lucie,chakraborty2024divide,park2024when}.

Despite their advantages and broad applications, invariant measures are known for the fundamental limitation that distinct systems can share the same invariant measure~\cite{Foreman2011}. Consequently, relying solely on a comparison of invariant measures may lead to incorrect identification of the underlying dynamics. Inspired by the delay-coordinate transformation from Takens' embedding theory~\cite{Takens1981DetectingSA,sauer1991embedology}, comparing \textbf{invariant measures in \textit{time-delay coordinates}} is a promising alternative for system identification. Time-delay coordinates are assembled by sampling a single observable at evenly spaced intervals, constructing vectors of the form $[x(t), x(t-\tau), \dots, x(t-(m-1)\tau)]$, where $\tau$ denotes the delay interval, $m$ is the embedding dimension, and $[x(t), y(t), z(t)]$ is the original state coordinate. Time-delay coordinates can be easily formed from partially observed data and have been widely employed to reconstruct full system dynamics \cite{farmer1987predicting,young2023deep}. As shown in Fig.~\ref{fig:torus}, systems sharing identical invariant measures in their original state coordinates can exhibit distinct invariant measures in time-delay coordinates. This motivates our theoretical and computational study on using invariant measures in time-delay coordinates for modeling dynamical systems.

\begin{figure}[h!]
    \centering
    \includegraphics[width = \linewidth]{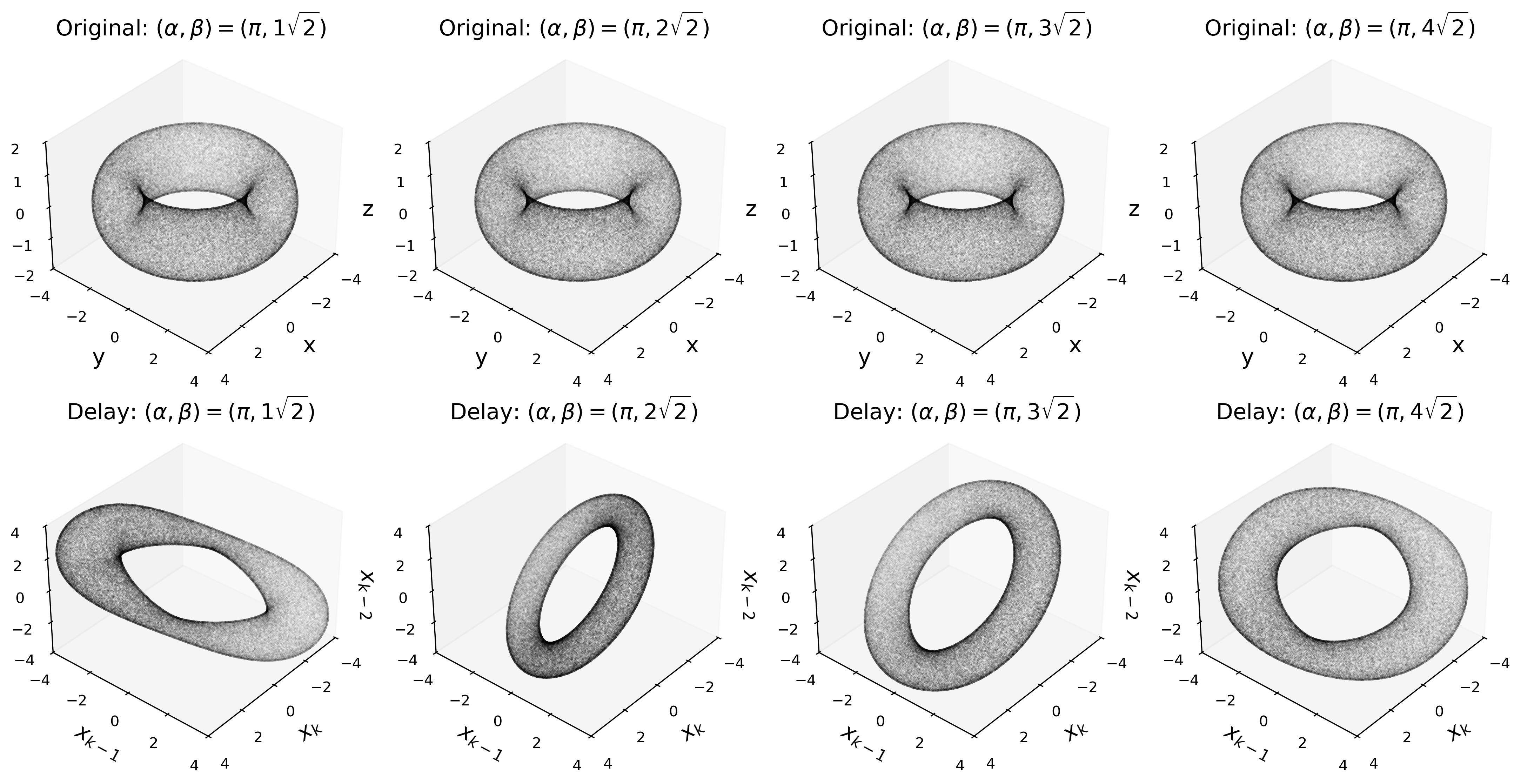}
    \vspace{-.5cm}
    \caption{Illustration of the difference between invariant measures in the state coordinates and the delay coordinates for the torus rotation $T_{\alpha,\beta}(z_1,z_2)= (z_1+\alpha,z_2+\beta) \pmod{ 1}$. The top row shows four different choices of $(\alpha,\beta)$ for which the state-coordinate invariant measures coincide. The bottom row shows that all four systems can be distinguished by their invariant measure in the time-delay coordinates.}
    \label{fig:torus}
\end{figure}

In this work, we present two key theoretical results: (1)~if two dynamical systems share the same invariant measure in time-delay coordinates, then they are topologically conjugate on the support of the measure; and (2)~the non-uniqueness inherent in this conjugacy can be resolved by utilizing \textit{multiple delay-coordinate invariant measures} obtained from distinct observables. Specifically, a finite set of delay-coordinate invariant measures provides sufficient information to uniquely identify a dynamical system under mild assumptions. Although previous studies have explored the numerical use of delay-coordinate invariant measures~\cite{MOECKEL1997187, greve2019data, botvinick2023learning}, our work is the first to establish formal theoretical guarantees. We further translate these theoretical insights into practical computational algorithms through the alignment of simulated and observed invariant measures in delay coordinates. We demonstrate the robustness of our algorithms using a series of physical examples under challenging data conditions.

\smallskip
\textit{State-coordinate invariant measures.}\----
We begin by introducing key notions that will appear in our subsequent discussion. Let $ \mathcal{P}(X) $ denote the space of probability measures on a domain $ X $. For any $ \mu \in \mathcal{P}(X) $, the \textit{support} of $ \mu $, denoted $ \mathrm{supp}(\mu) \subseteq X $, is the subset of $ X $ where $ \mu $ assigns nonzero probability. A fundamental concept in the transformation of measures is the \textit{pushforward}, which describes how one probability measure is mapped to another under a given function. Specifically, for a mapping $f: X \to Y$, the pushforward of a measure $\mu \in \mathcal{P}(X)$, denoted by $f\#\mu \in \mathcal{P}(Y)$, is defined as: 
$$
(f\#\mu)(B) = \mu(f^{-1}(B)) \quad \forall B \subseteq Y,
$$
for any suitable subset $B$. Intuitively, the pushforward captures the redistribution of probabilities induced by the action of $ f $. Thus, $f$ can be seen as a coordinate transform or a ``change of variables.''

Consider now a discrete-time dynamical system defined by a mapping $ T: X \to X $. For continuous-time dynamical systems, $ T $ may correspond to the time-$ \Delta t $ flow map of a vector field. A measure $ \mu \in \mathcal{P}(X) $ is said to be \textit{$ T $-invariant} if $ T\#\mu = \mu $, signifying that $ \mu $ remains unchanged under the action of $ T $. Invariant measures, therefore, encode the statistical properties of the system that remain stable over time.

An invariant measure is further classified as \textit{$ T $-ergodic} if $ \mu(B) \in \{0, 1\} $ for any subset $B$ where $T^{-1}(B) = B$, i.e., $B$ is $T$-invariant. This implies that nontrivial $T$-invariant sets  have either zero or full measure. Ergodicity ensures that the system cannot be decomposed into smaller independent subsystems.

To explicitly connect these concepts to observable dynamics, we introduce the \textit{basin of attraction} of an invariant measure $\mu$, denoted by $B_{\mu,T} \subseteq X$. The basin consists of all initial states $x \in X$ whose long-term, time-averaged behavior aligns with the measure $\mu$. Formally, for any observable $\phi \in C(X)$, a state $x$ lies in the basin if
\begin{equation}\label{eq:basin}
\lim_{N\to\infty} \frac{1}{N}\sum_{k=0}^{N-1}\phi(T^k(x)) = \int \phi\,d\mu, \quad \forall \phi \in C(X).
\end{equation}
A measure $\mu$ is called a \textit{physical measure} if its basin of attraction $B_{\mu,T}$ has positive Lebesgue measure~\cite{Young2002}, meaning that the basin of attraction occupies a non-negligible volume in $X$. Prominent examples of dynamical systems admitting physical measures include attracting fixed points, stable limit cycles, and chaotic attractors, such as the Lorenz-63 system~\cite{luzzatto2005lorenz}. Physical measures are essential because they characterize the statistical properties of trajectories that are observable in practice, bridging the gap between mathematical theory and experimental observations.

\smallskip 

\textit{Delay-coordinate invariant measures.}\----  
We now examine how invariant measures are transformed under a change of coordinates induced by the dynamics. Consider two dynamical systems, $T: X \to X$ and $\widetilde{T}: Y \to Y$. These systems are said to be \textit{topologically conjugate} if there exists an invertible continuous function $h: X \to Y$ with a continuous inverse, such that $\widetilde{T} = h \circ T \circ h^{-1}$. Physically, $h$ can be interpreted as a nonlinear coordinate transformation that maps the trajectories of $T$ onto those of $\widetilde{T}$, thereby re-expressing the system's behavior in a new coordinate frame.

 The following result, rigorously proven in the supplemental materials (\textbf{SM}), establishes a clear connection between the statistical properties of two systems that are equivalent under a coordinate transformation.
\begin{proposition}\label{prop:1}
If $\widetilde{T} = h\circ T \circ h^{-1}$ and $\mu$ is $T$-invariant, then $h\# \mu$ is $\widetilde{T}$-invariant. Moreover, if $\mu$ is $T$-ergodic, then $h\# \mu$ is additionally $\widetilde{T}$-ergodic.
\end{proposition}

Although invariant measures alone may not suffice to distinguish between two systems $T,S:X\to X$, Proposition~\ref{prop:1} provides an important insight: it may be possible to achieve this distinction in transformed coordinates. To this end, we construct maps $h_T, h_S: X \to Y$, which transform $T$ and $S$ into conjugate systems $\widetilde{T} = h_T \circ T \circ h_T^{-1}$ and $\widetilde{S} = h_S \circ S \circ h_S^{-1}$ acting on $Y$. The invariant measures of these conjugate systems, $h_T\#\mu \in \mathcal{P}(Y)$ and $h_S\#\mu \in \mathcal{P}(Y)$, now become the objects of comparison. In what follows, we take $h_T$ and $h_S$ to be \textit{delay-coordinate maps}, inspired by Takens' seminal embedding theory~\cite{Takens1981DetectingSA}.

Time-delay embedding has revolutionized the study of nonlinear trajectory data, with applications across diverse fields, including fluid mechanics and neuroscience~\cite{saad1998comparative,tajima2015untangling,yuan2021flow,young2023deep}. In many practical scenarios, the full trajectory of a dynamical system, $ \{T^k(x)\}_{k=0}^{N} $, is unavailable. Instead, experimental recordings often consist of a scalar time series, $ \{y(T^k(x))\}_{k=0}^{N} $, where $ y $ is a scalar-valued observable capturing partial measurements of the system. Remarkably, time-delay embedding can reconstruct the original dynamical system—up to topological conjugacy—using only this scalar time-series data.

Central to this approach is the \textit{delay-coordinate map}
\begin{equation}\label{eq:delay_coord}
    \Psi_{(y,T)}^{(m)}(x) := \left[y(x), y\left(T(x)\right), \dots, y\left(T^{m-1}(x)\right)\right],
\end{equation}
which is injective under proper  conditions~\cite{Takens1981DetectingSA,sauer1991embedology}. Physically, this means that the delay-coordinate map can uniquely distinguish points in $X$, even when only a single coordinate of $X$ is recorded over time. This powerful property enables the reconstruction of the system's dynamics in the delay-coordinate space $\mathbb{R}^m$, offering insight into the underlying  system $T:X\rightarrow X$ that would otherwise remain hidden.

Using the delay-coordinate map, the transformed dynamics can be expressed as:
\begin{equation}\label{eq:delay_dynamics}
    \hat{T}_{(y,m)} := \big[\Psi_{(y,T)}^{(m)}\big]\circ T \circ \big[\Psi_{(y,T)}^{(m)}\big]^{-1}: \mathbb{R}^m\rightarrow \mathbb{R}^m.
\end{equation}
By~\eqref{eq:delay_dynamics},  $T$ and $\hat{T}_{(y,m)}$ are topologically conjugate. Consequently, if $T$ has an invariant measure $\mu \in \mathcal{P}(X)$, then $\hat{T}_{(y,m)}$ admits the invariant measure
\begin{equation}\label{eq:delay_im}
    \hat{\mu}_{(y,T)}^{(m)} := \Psi_{(y,T)}^{(m)}\#\mu \in \mathcal{P}(\mathbb{R}^m)\,,
\end{equation}
as a result of~Proposition~\ref{prop:1}. This measure, referred to as the \textit{delay-coordinate invariant measure}, provides a statistical characterization of the reconstructed dynamics in $ \mathbb{R}^m $. We propose using this measure for system identification tasks.  See Fig.~\ref{fig:lorenz_visual} for a visualization.

\begin{figure}[h]
    \centering
    \includegraphics[width=\linewidth]{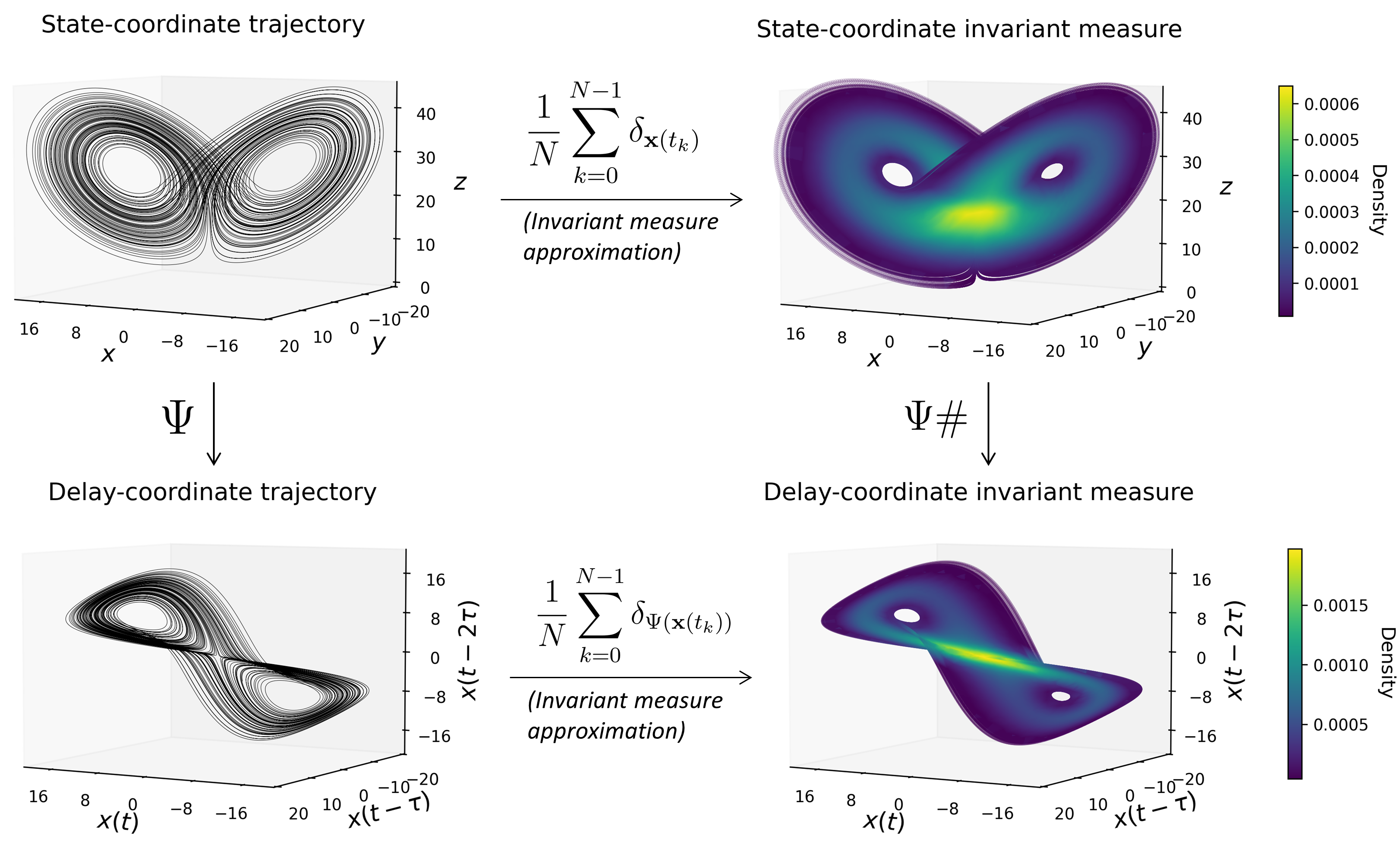}
    \vspace{-.5cm}
    \caption{Visualization of the state-coordinate (top row) and delay-coordinate (bottom row) invariant measures for the Lorenz-63 system. Here, $\Psi$ is the delay-coordinate map~\eqref{eq:delay_coord}. }
    \label{fig:lorenz_visual}
\end{figure}
\smallskip

\textit{Theoretical Results.}\---- 
We now present our primary theoretical results. For clarity, we provide a concise overview of the theorems here, while the complete versions and detailed proofs are available in the  \textbf{SM}.

Consider two smooth discrete-time dynamical systems, each defined by an invertible, differentiable map $T$ or $S$ on an open set $U \subseteq \mathbb{R}^n$, and each admitting an invariant measure $\mu$ and $\nu$, respectively. We choose an embedding dimension $m \in \mathbb{N}$ such that $m > 2\max\{d_{\mu}, d_{\nu}\}$, where $d_{\mu}$ is the box-counting (or fractal) dimension~\cite{liebovitch1989fast} of the attractor associated with $\mu$. For example, for the Lorenz-63 system shown in Fig.~\ref{fig:lorenz_visual}, one typically finds $d_{\mu} \approx 2$.  We assume a mild condition on the growth rate of periodic orbits for $T$ and $S$, typically satisfied in physical systems. Under this setup, we present our first result below.
\begin{theorem}[Sketch]\label{thm:main1}
Equality $\hat{\mu}_{(y,T)}^{(m+1)} = \hat{\nu}_{(y,S)}^{(m+1)}$ of two invariant measures in time-delay coordinates  implies  $T$ and $S$ are topologically conjugate on $\textup{supp}(\mu)$, for almost every continuous and differentiable observable $y$.
\end{theorem}

The phrase ``almost every'' in Theorem~\ref{thm:main1} comes from the theory of prevalence~\cite{hunt1992prevalence}, implying that the set of observables $y$ for which the result does not hold is negligible in a probabilistic sense. The embedding dimension required to construct the delay-coordinate invariant measures in Theorem~\ref{thm:main1} is one higher than the dimension specified in the generalized Takens' theorem~\cite{sauer1991embedology}, which  is critical in our analysis; see the \textbf{SM}. Indeed, the key to our proof is the observation that a single point 
\begin{equation*}\label{eq:point_ex}
    (\lefteqn{\underbrace{\phantom{y(x),y(T(x)),\dots, y(T^{m-1}(x))}}_{\footnotesize\Psi_{(y,T)}^{(m)}(x)}}y(x),\overbrace{y(T(x)),\dots, y(T^{m-1}(x)), y(T^m(x))}^{\footnotesize\Psi_{(y,T)}^{(m)}(T(x))})\in\mathbb{R}^{m+1}
\end{equation*} 
determines both $\Psi_{(y,T)}^{(m)}(x)$ and $\Psi_{(y,T)}^{(m)}(T(x))$, which forms an input-output pair of the $m$-dimensional delay-coordinate dynamics defined by $\hat{T}_{(y,m)}$; see \eqref{eq:delay_dynamics}.  

Existing works that perform system identification by comparing state-coordinate invariant measures often struggle to enforce ergodicity of the reconstructed system~\cite{greve2019data, yang2023optimal, botvinick2023learning}. In contrast, the conjugacy relation deduced in Theorem~\ref{thm:main1} implies that aligning delay-coordinate invariant measures can \textit{guarantee ergodicity for the reconstructed dynamics}.

\begin{corollary}[Sketch]\label{cor:ergodic}
If $T$ is $\mu$-ergodic, then the equality of delay-coordinate invariant measures in Theorem~\ref{thm:main1} implies that $S$ is additionally $\nu$-ergodic, for almost every continuous and differentiable observable.
\end{corollary}
While Theorem~\ref{thm:main1} showed that a single observable can reveal topological equivalence, our second main result goes further: by using a small number of observables, along with short trajectory information from one well-chosen initial condition, we can \textit{uniquely recover the dynamics on the support of $\mu$}. In other words, this richer set of delay-coordinate measures is enough to distinguish the two systems.  
\begin{theorem}[Sketch]\label{thm:main2}
Let systems $T$ and $S$ share the same invariant measure $\mu$. Assume
\begin{enumerate}
    \item there exists a point $x^* \in B_{\mu,T} \cap \textup{supp}(\mu)$ such that $T^k(x^*) = S^k(x^*)$ for $1 \leq k \leq m - 1$, and \label{item:assumption1}
    \item the delay-coordinate invariant measures built from $m$ different observables, $\{y_i\}_{i=1}^m$, are identical under both $T$ and $S$.\label{item:assumption2}
\end{enumerate}
Then $T =S$ on $\textup{supp}(\mu)$, for almost every continuous and differentiable observable $Y = (y_1, \dots, y_m)$.
\end{theorem}

In the \textbf{SM}, we also discuss a variant of Theorem~\ref{thm:main2}, which provides an additional method for obtaining unique system identification from a measure-based comparison.\smallskip

\textit{Algorithms.}\---- Motivated by Theorems~\ref{thm:main1} and~\ref{thm:main2}, we now propose practical algorithms for data-driven system identification based on the comparison of simulated and observed delay-coordinate invariant measures. Specifically, we assume access to a parameterized model $T_\theta : \mathbb{R}^n \to \mathbb{R}^n$, which we seek to calibrate to the true dynamics using observed trajectory data. For the examples considered, all systems operate in continuous time, and thus $T_\theta$ is defined as the flow map of a parameter-dependent differential equation. In Algorithm~\ref{alg:1}, we outline a method for recovering these unknown parameters from partially observed trajectory data using the delay-coordinate invariant measure.\smallskip

\begin{algorithm}[H] \label{alg:1}
\caption{\protect\mbox{Trajectory-based delay measure opt.}}
\SetAlgoLined
\noindent \textbf{Input.} Partial observations $\{y(t_i)\}_{i=1}^N\subseteq \mathbb{R}$.\\
\noindent (i) Choose embedding parameters $m,\overline{\tau} \in \mathbb{N}$.\\
(ii) Set $K = N-(m-1)\overline{\tau}$ and construct \\
\vspace{-.6cm}
\begin{equation}\label{eq:delayconstr}
    \hat{\mu}_{y} = \frac{1}{K }\sum_{i=1}^K \delta_{(y(t_{i+(m-1)\overline{\tau}}),\dots,  y(t_i))} \in \mathcal{P}(\mathbb{R}^m).
\end{equation}
\vspace{-.5cm}
\\
(iii) Solve the optimization\\
\vspace{-0.6cm}
$$\tilde{\theta} = \argmin_{\theta\in \Theta} \mathcal{D}(\hat{\mu}_{y}(\theta), \hat{\mu}_y),$$\\
\vspace{-.3cm}
where $\hat{\mu}_{y}(\theta)$ is computed from a long trajectory \\
\noindent simulation with parameter $\theta\in\Theta$, following \eqref{eq:delayconstr}.\\
\noindent \textbf{Output.} Optimal parameters $\tilde{\theta} \in \Theta$.
\end{algorithm}
\smallskip
In Algorithm~\ref{alg:1}, $\mathcal{D}(\cdot,\cdot)$ is any metric or divergence on the space of probability measures, e.g., Maximum-Mean Discrepancy (MMD) or Wasserstein distance. While Algorithm~\ref{alg:1} is compatible with partially observed trajectory data, it can be challenging to obtain the gradients of $\theta \mapsto \hat{\mu}_y(\theta)$, due to the approximation of $\hat{\mu}_y(\theta)$ from a long trajectory simulation. However, when the dimension of $\theta$ is large, it is necessary to use gradient-based optimization. In Algorithm~\ref{alg:2}, we introduce an alternative approach to delay-coordinate invariant measure matching which enables stable gradient computations.  \smallskip

\begin{algorithm} \label{alg:2}
\caption{\protect\mbox{Pushforward-based delay measure opt.} }
\SetAlgoLined
\noindent\textbf{Input.} Full-state observations $\{x(t_i)\}_{i=1}^N\subseteq \mathbb{R}^{n}$.\\
\noindent (i) Select $m,\overline{\tau} \in \mathbb{N}$ and observables $y_1,\dots, y_{\ell}$.\\
(ii) Construct $\hat{\mu}_{y_1},\dots, \hat{\mu}_{y_{\ell}}\in \mathcal{P}(\mathbb{R}^m)$ following \eqref{eq:delayconstr}.\\
(iii) Solve the optimization \\ \vspace{-.7cm}
$$\tilde{\theta} = \argmin_{\theta\in\Theta} \bigg[\mathcal{D}(T_{\theta}\#\mu,\mu)+\sum_{i=1}^{\ell} \mathcal{D}(\Psi_{(y_i,T_{\theta})}\#\mu, \hat{\mu}_{y_i})\bigg],$$ \vspace{-.5cm}\\
\noindent where $\mu$ is the measure extracted from $\{x(t_i)\}_{i=1}^N$\\ \noindent and $\Psi_{(y_i,T_{\theta})}$ is the delay map with parameter $\theta$.\\
\noindent\textbf{Output.} Optimal parameters $\tilde{\theta} \in \Theta$.
\end{algorithm}\smallskip

While Algorithm~\ref{alg:1} is designed to handle partial observations and relies on long trajectory simulations to approximate the delay-coordinate invariant measure, Algorithm~\ref{alg:2} assumes access to full-state observations for approximating $\mu$ but does not require long trajectory simulations. Consequently, Algorithm~\ref{alg:2} works stably with gradient-based optimization techniques.

To achieve unique system recovery using Algorithm~\ref{alg:2}, one can additionally incorporate information about the initial condition of the trajectory into the objective function, following Assumption~\ref{item:assumption1} in Theorem~\ref{thm:main2}. For additional details on the implementation of Algorithms~\ref{alg:1} and~\ref{alg:2}, we refer to the \textbf{SM}. \smallskip

\textit{Numerical Experiments.}\----  We present several numerical experiments that demonstrate the effectiveness of Algorithms~\ref{alg:1} and~\ref{alg:2} in modeling sparsely sampled, partially observed, noisy, high-dimensional, and chaotic physical systems. We summarize our findings below, with detailed experimental setups provided in the \textbf{SM}. We also include a GitHub repository~\cite{code} containing tutorials for the implementation of Algorithms \ref{alg:1} and \ref{alg:2} and the complete code for our numerical experiments.

Consider the parameterized Kuramoto--Sivashinsky (KS) equation $u_{t} + \theta(u_{xx} + u_{xxxx}) + uu_x = 0$, where $\theta\in [0.5,1.5].$ In Fig.~\ref{fig:KS}, we use Algorithm~\ref{alg:1} to infer the true parameter $\theta^* = 1$ from partially observed, noisy, and slowly sampled time series. The chaotic divergence of nearby trajectories in the KS equation causes significant difficulties when using an objective that relies on pointwise trajectory matching for the system comparison. In contrast,  our delay-coordinate invariant measure matching approach is effective at inferring the true $\theta$.
\begin{figure}
    \centering
    \includegraphics[width=\linewidth]{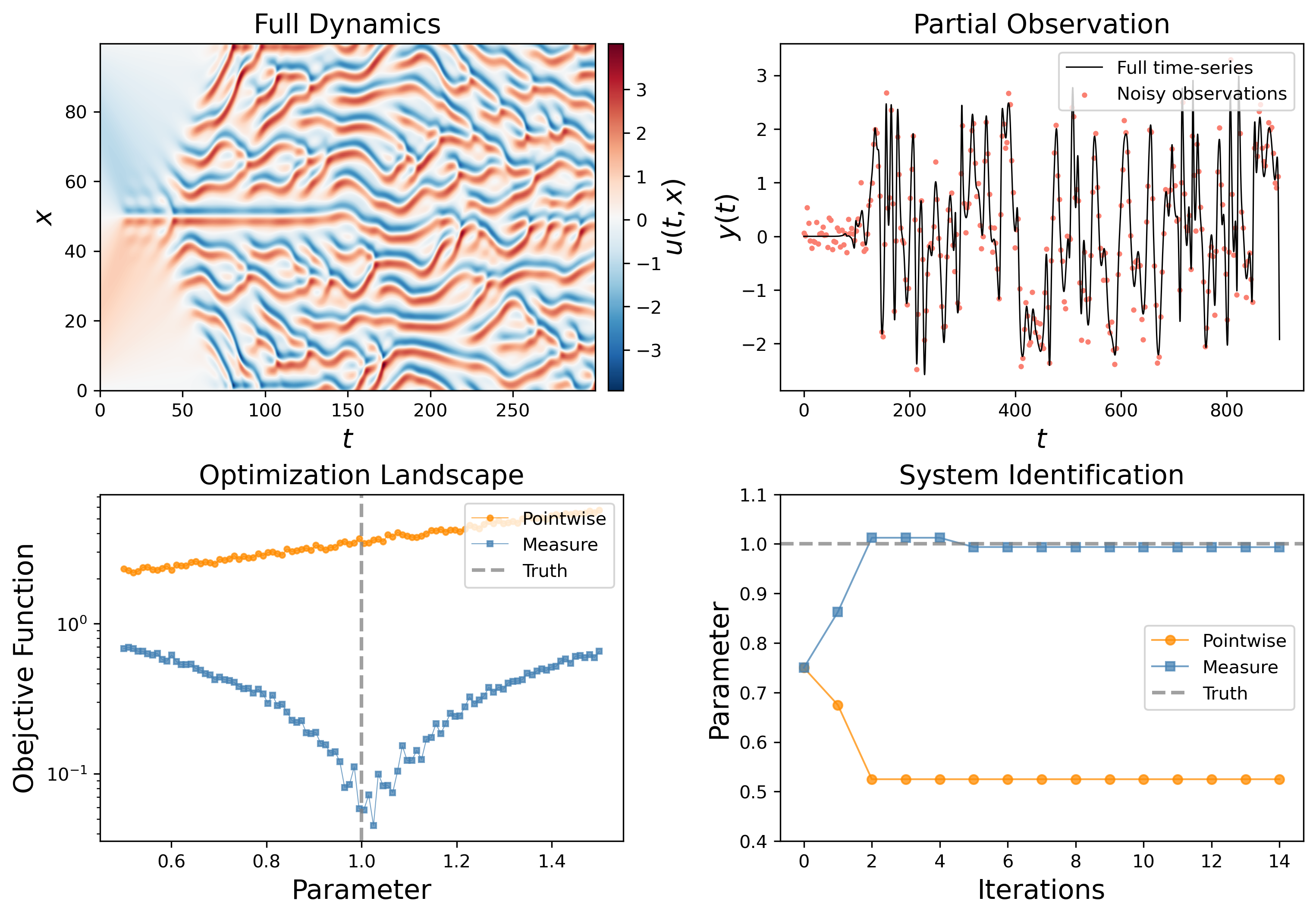}
    \caption{Parameter identification for the KS equation using a delay-coordinate invariant measure objective. Top left: Full, unobserved dynamical system. Top right: The noisy, partially observed inference data. Bottom left: Optimization landscapes of the delay-coordinate invariant measure and pointwise objective functions. Bottom right: Nelder--Mead iterations for minimizing the two objective functions.}
    \label{fig:KS}
\end{figure}
\begin{figure}
      \centering
    \includegraphics[width=\linewidth]{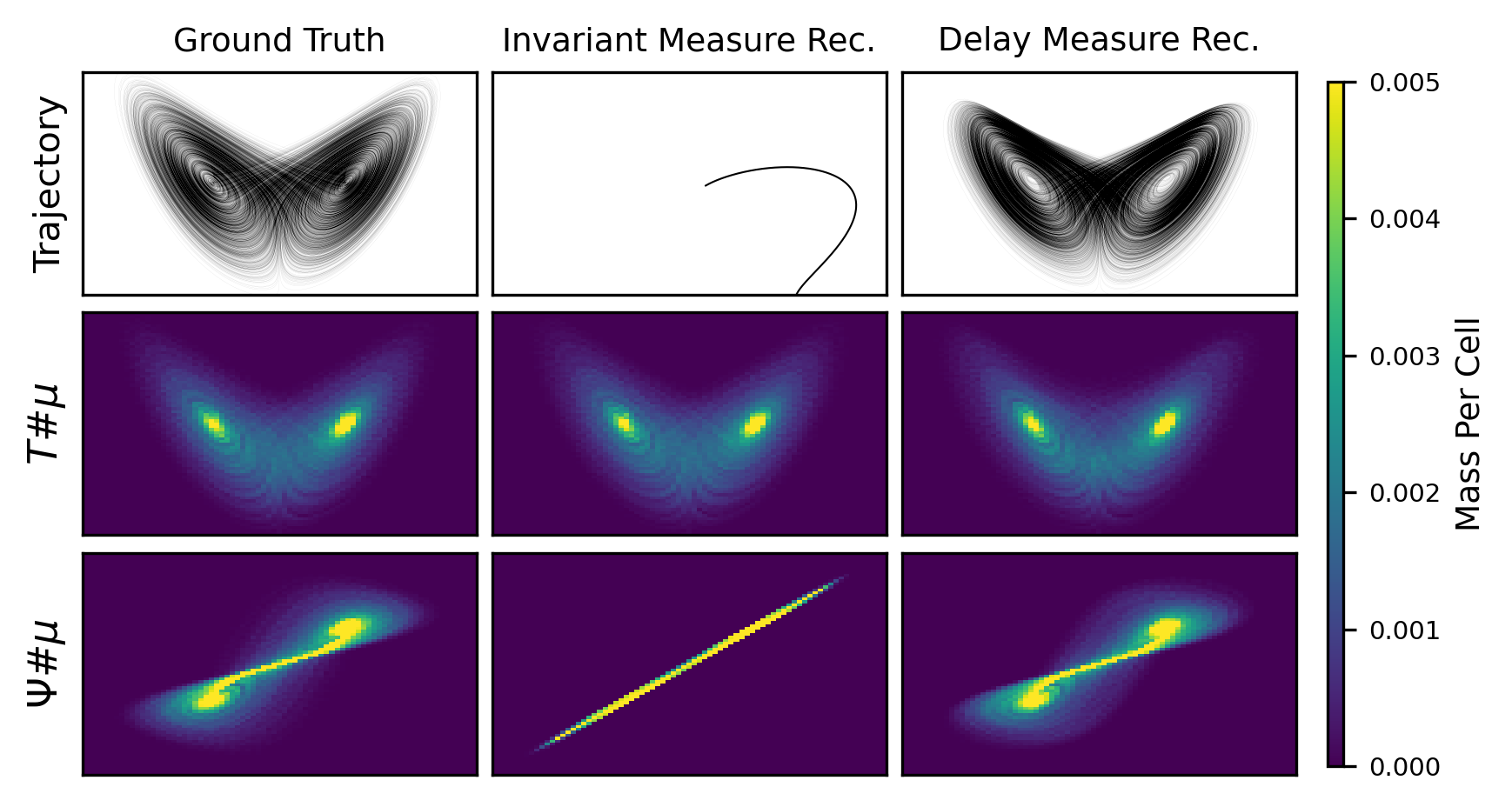}
    \caption{Reconstructing the dynamics of the Lorenz-63 attractor (left column) using invariant measure-based loss functions. A loss function based solely on the state-coordinate invariant measure proves insufficient for reconstructing the dynamics (middle column), whereas the objective introduced in Algorithm~\ref{alg:2} enables an accurate reconstruction (right column).}
    \label{fig:lorenz_compare}
\end{figure}
\begin{figure}
      \centering\hspace{.7cm}
    \includegraphics[width=.9\linewidth,trim={0 0 0 2mm},clip]{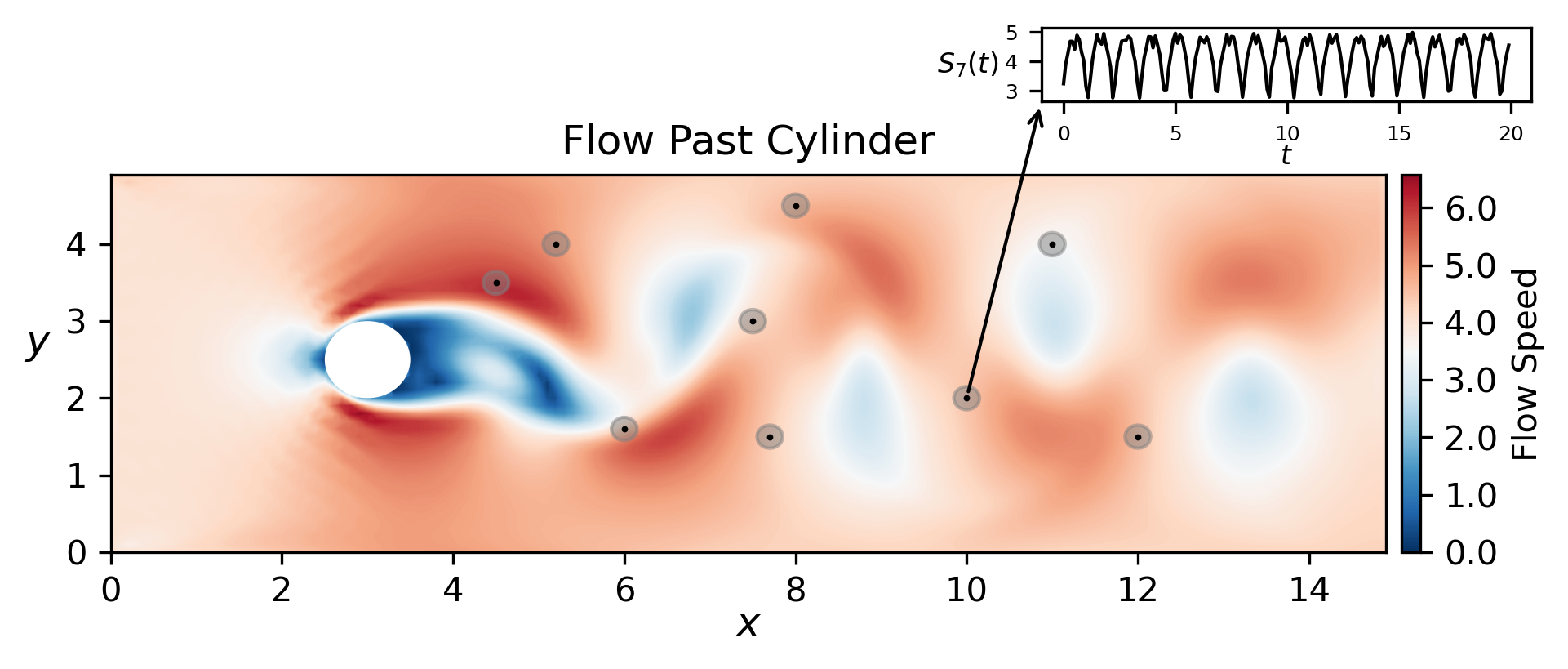}\\
    \vspace{-.1cm}
    \includegraphics[width = .85\linewidth,trim={0 3mm 0 0},clip]{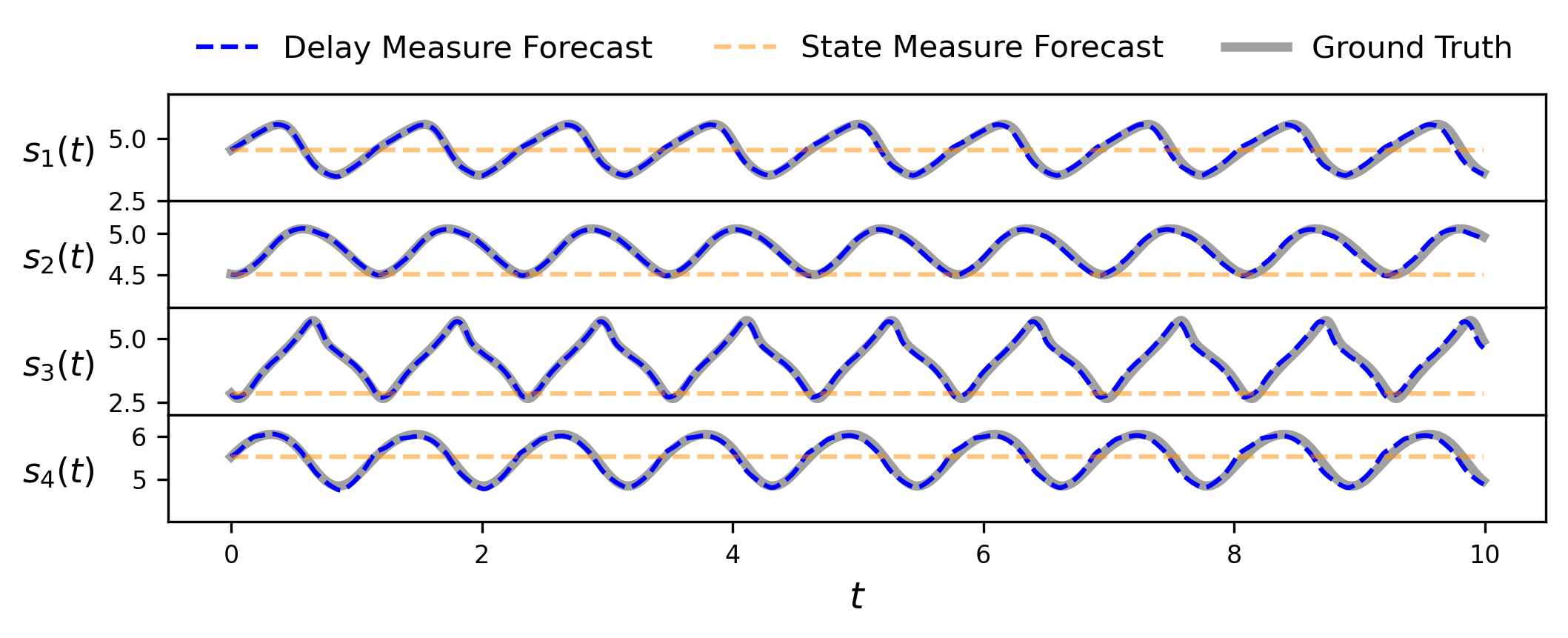}
    \caption{Forecasting the dynamics of noisy observables in a K\'arm\'an vortex street using delay-coordinate invariant measures. Top: Snapshot of the flow and illustration of the partially observed velocity probes. Bottom: Forecast using both state-coordinate and delay-coordinate invariant measures.}
    \label{fig:cylinder}
\end{figure}

We next consider an example involving the canonical Lorenz-63 system, in which  $T_\theta$ is defined as the time-$\Delta t$ flow map of a neural network-parameterized velocity field with $>10^4$ unknown parameters. Due to the high-dimensional parameter space, gradient-based methods are necessary to perform optimization. In Fig.~\ref{fig:lorenz_compare}, we demonstrate the effectiveness of Algorithm~\ref{alg:2} for learning the Lorenz-63 dynamics using delay-coordinate invariant measures. An analogous method based solely on state-coordinate invariant measure matching converges only to the identity map, failing to capture the true dynamics of the attractor. In contrast, our delay-measure loss naturally enforces ergodicity of the reconstructed dynamics (see Corollary~\ref{cor:ergodic}).

Our final example models partial observations of flow past a cylinder at a Reynolds number $\mathrm{Re} = 70$, which exhibits vortex shedding and is subject to challenging data assumptions. A sparse dataset consisting of only $N = 200$ observations is sampled from 7 probes in the K\'arm\'an vortex street, where the probe locations are subject to Gaussian deviations, introducing uncertainty. Using Algorithm~\ref{alg:2}, we construct a neural network model for the observed noisy dynamics, which we assume to collectively make up the full state of a new dynamical system; see Fig.~\ref{fig:cylinder}. Similar to the example shown in Fig.~\ref{fig:lorenz_compare}, modeling the dynamics by matching only the state-coordinate invariant measure fails to capture the true dynamics. Incorporating the trajectory initial condition (see Assumption \ref{item:assumption1}) and delay-coordinate measures matching (see Assumption~\ref{item:assumption2}) in the optimization process is necessary to achieve accurate predictions of the fluid flow sensors.

\smallskip

\textit{Conclusion.}\---- In recent years, measure-theoretic approaches have emerged as essential tools for modeling dynamical systems from imperfect data subject to sparsity, noise, and uncertainty. While many works have leveraged the use of the state-coordinate invariant measure to perform system identification~\cite{botvinick2023learning, yang2023optimal, greve2019data, jiang2023training, schiff2024dyslim, chen2023detecting}, the invariant measure alone cannot distinguish between large classes of models which all share the same asymptotic behavior. To this end, we have developed a theoretical framework that justifies the use of invariant measures in time-delay coordinates for comparing dynamical systems. Our approach retains the benefits of a measure-theoretic approach, e.g., robustness to noise and chaos, while also ensuring the unique identifiability of the underlying dynamical system, up to a topological conjugacy. Moreover, it is computationally cheap and straightforward to form the delay-coordinate invariant measure from partially observed time-series data, making our approach widely applicable across a range of physical applications. We expect our theoretical advances will inspire future work leveraging invariant measures in time-delay coordinates to perform more effective system identification in the face of data uncertainty and noise.

\smallskip
\textit{Acknowledgments.}\----J.~B.-G.~was supported by a fellowship award under contract FA9550-21-F-0003
through the National Defense Science and Engineering Graduate (NDSEG) Fellowship Program,
sponsored by the Air Force Research Laboratory (AFRL), the Office of Naval Research (ONR) and
the Army Research Office. Y.~Y. acknowledges support from National Science Foundation (NSF) grant DMS-2409855 and ONR grant
N00014-24-1-2088. R.~M.~acknowledges support for initial problem concept formulation under Air Force Office of Scientific Research (AFOSR) grant No.~FA9550-23RQCOR007.

\bibliography{bibliography_arxiv}%

\newgeometry{top=1in, bottom=1in, left=1in, right=1in}

\onecolumngrid
\renewcommand{\baselinestretch}{1.5}\normalsize

\renewcommand{\thetheorem}{S\arabic{section}.\arabic{theorem}}

\newtheorem{lemma}{Lemma}[section]

\newtheorem{assumption}{Assumption}[section]
\theoremstyle{definition}
\newtheorem{remark}{Remark}[section]
\newtheorem{definition}{Definition}[section]

\titlespacing*{\section}
{0pt}{1em plus 0.2em minus 0.2em}{0.5em}       %
\titlespacing*{\subsection}
{0pt}{0.8em plus 0.2em minus 0.2em}{0.4em}
\titlespacing*{\subsubsection}
{0pt}{0.6em plus 0.1em minus 0.1em}{0.3em}

\renewcommand{\theproposition}{S\arabic{section}.\arabic{proposition}}
\renewcommand{\thecorollary}{S\arabic{section}.\arabic{corollary}}
\renewcommand{\theremark}{S\arabic{section}.\arabic{remark}}
\renewcommand{\theassumption}{S\arabic{section}.\arabic{assumption}}
\renewcommand{\thedefinition}{S\arabic{section}.\arabic{definition}}
\renewcommand{\thelemma}{S\arabic{section}.\arabic{lemma}}
\renewcommand{\theequation}{S\arabic{equation}}
\setcounter{section}{0}
\setcounter{theorem}{0}
\setcounter{subsection}{0}
\setcounter{proposition}{0}
\setcounter{subsubsection}{0}
\setcounter{equation}{0}
\setcounter{theorem}{0}
\setcounter{lemma}{0}
\setcounter{remark}{0}
\setcounter{definition}{0}
\setcounter{corollary}{0}

\makeatletter
\renewcommand{\thesection}{\arabic{section}}
\renewcommand{\thesubsection}{\arabic{section}.\arabic{subsection}}
\renewcommand{\thesubsubsection}{\arabic{section}.\arabic{subsection}.\arabic{subsubsection}}
\renewcommand{\p@subsection}{}
\renewcommand{\p@subsubsection}{}

\makeatother

\titleformat{\section}{\normalfont\Large\bfseries}{\thesection}{1em}{} 
\titleformat{\subsection}{\normalfont\large\bfseries}{\thesubsection}{1em}{} 
\titleformat{\subsubsection}{\normalfont\normalsize\bfseries}{\thesubsubsection}
{1em}{} 
\makeatother
\newpage
{ \begin{center}
    \huge Supplemental Materials
\end{center}}
\section{Introduction}
In the main text, we presented two central theorems which motivate the use of delay-coordinate invariant measures for performing data-driven system identification. We then used this theory to motivate practical computational algorithms based upon the comparison of simulated and observed delay-coordinate invariant measure, which we deployed across several physical examples. This supplemental document provides additional detail on the theory presented in the main text, the implementation of our computational algorithms, and the setup of our numerical experiments. The organization of the supplemental materials is as follows. In Section \ref{subsec:notation}, we introduce rigorous definitions and background material required for the proofs of our theorems appearing in the main text. In Section~\ref{sec:results}, we then present the full statements and proofs of these results. Finally, in Section \ref{sec:full_exp} we provide additional details on the algorithms and numerical experiments presented in the main text.

\section{Background and Preliminaries}\label{subsec:notation}
This section reviews the essential background and preliminary results necessary for the statements and proofs of our main results.  In Section~\ref{subsec:background1}, we cover pushforward measures, invariant measures, and basins of attraction. In Section~\ref{subsec:delay_measures}, we introduce the time-delay map and the corresponding invariant measure in time-delay coordinates, which we propose to use for system identification. Finally, in Section~\ref{subsec:embedding}, we discuss generalizations of the Whitney and Takens embedding theorems, which ensure that the invariant measure in time-delay coordinates is well-defined. 
While the content presented in Section \ref{subsec:background1} and Section \ref{subsec:delay_measures} shares some overlap with the main text, the formulations we present here are fully rigorous, while the main text instead contains brief sketches which omit detail.

\subsection{Invariant Measures}\label{subsec:background1}
Throughout our proofs, it will be important to reference the support of a given probability measure, which can intuitively be viewed as the region on which the measure is concentrated; see \cite[Theorem 2.1]{parthasarathy2005probability}. 
\begin{definition}[Support of a measure]\label{def:support}
    If $X$ is a Polish space and $\mu\in \mathcal{P}(X)$ is a Borel probability measure, then the \textit{support} of $\mu$ is the unique closed set $\text{supp}(\mu)\subseteq X$ such that $\mu(\text{supp}(\mu)) = 1$, and for any other closed set $C\subseteq X$ satisfying $\mu(C) = 1$, it holds that $\text{supp}(\mu)\subseteq C$.
\end{definition}
We will frequently need to discuss how probability measures are transformed under the action of a measurable function. This concept is formalized with the notion of a pushforward measure. 
\begin{definition}[Pushforward measure]
    Let $X$ and $Y$ be Polish spaces, let $f:X\to Y$ be Borel measurable, and let $\mu\in \mathcal{P}(X).$ Then, the \textit{pushforward} of $\mu$ is the probability measure $f\# \mu \in \mathcal{P}(Y)$, defined by $(f\# \mu)(B) := \mu(f^{-1}(B))$, for all Borel sets $B\subseteq Y$. 
\end{definition}
We next define the notion of an invariant measure, which is a probability measure that remains unchanged under the pushforward of a given self-transformation $T:X\to X$. In the study of dynamical systems, application of the self-transformation $T$ to a state $x\in X$, i.e., $x\mapsto T(x)$, represents the evolution of time. The invariant measure then provides a statistical description of the system's asymptotic behavior.
\begin{definition}[Invariant measure]
    Let $X$ be a Polish space and let $T:X\to X$ be Borel measurable. The probability measure $\mu \in \mathcal{P}(X)$ is said to be \textit{$T$-invariant} (or simply \textit{invariant} when the map $T$ is clear from context) if $T\# \mu = \mu$. 
\end{definition}
Given a dynamical system $T:X\to X$, it is natural to ask which initial conditions $x\in X$ produce trajectories $\{T^k(x)\}_{k\in\mathbb{N}}$ that yield empirical measures which are asymptotic to a given invariant measure. Given that experimentalists primarily observe trajectories of dynamical systems, it is important to understand the extent to which this data provides information about the underlying invariant measure. This leads us to define the basin of attraction.
\begin{definition}[Basin of attraction]\label{def:basin}
    Let $X$ be a Polish space, and let $T:X\to X$ be Borel measurable. The \textit{basin of attraction} of a $T$-invariant measure $\mu\in \mathcal{P}(X)$ with compact support is
   \begin{equation}\label{eq:basin}
       \mathcal{B}_{\mu,T} := \Bigg\{x\in X: \lim_{N\to\infty}\frac{1}{N}\sum_{k=0}^{N-1} \phi(T^k(x)) = \int_X \phi d\mu, \, \forall \phi \in C(X) \Bigg\}.
   \end{equation} 
\end{definition}

\subsection{Invariant Measures in Time-Delay Coordinates}\label{subsec:delay_measures}
Several times throughout this work, we study and compare dynamical systems in transformed coordinate systems. A fundamental notion of equivalence between the dynamical systems $T:X\to X$ and $\widetilde{T}:Y\to Y$ is given by the concept of topological conjugacy.
\begin{definition}[Topological conjugacy]\label{def:conjugacy}
    Let $X$ and $Y$ be Polish spaces, and let $T:X\to X$ and $\widetilde{T}:Y\to Y$ be continuous. The maps $T$ and $\widetilde{T}$ are said to be \textit{topologically conjugate} if there exists a homeomorphism $h: X\to Y$, such that $\widetilde{T} = h\circ T \circ h^{-1}$.
\end{definition}
In Definition \ref{def:conjugacy}, the homeomorphism $h:X\to Y$ is known as the conjugating map. When $T:X\to X$ admits $\mu \in \mathcal{P}(X)$ as an invariant measure, the following proposition, which appears in the main text as Proposition 1, characterizes an invariant measure of the conjugate system $\widetilde{T} = h\circ T \circ h^{-1}$ via a pushforward measure. For completeness, we provide a short proof of this fact. 
\begin{proposition}[Invariant measures under conjugacy]\label{prop:conjugacy}
    Let $X$ and $Y$ be Polish spaces, let $T:X\to X$ be continuous, and let $\mu \in \mathcal{P}(X)$ be $T$-invariant. If $h: X\to Y$ is a homeomorphism, then $h\# \mu\in \mathcal{P}(Y)$ is $\widetilde{T}$-invariant, where $\widetilde{T} = h\circ T \circ h^{-1}$. Moreover, $\mu$ is $T$-ergodic, then $h\# \mu$ is $\widetilde{T}$-ergodic.
\end{proposition}
\begin{proof}
    First, note that for any Borel subset $B\subseteq Y$ it holds that
    \begin{align*}
        (h\# \mu)(\widetilde{T}^{-1}(B)) = \mu(h^{-1}(\widetilde{T}^{-1}(B)) ) &= \mu (h^{-1}(h(T^{-1}(h^{-1}(B))))) \\
        &= \mu (T^{-1}(h^{-1}(B)))= \mu(h^{-1}(B))= (h\#\mu)(B),
    \end{align*}
    and thus $h\#\mu$ is $\widetilde{T}$-invariant. Furthermore, if $\widetilde{T}^{-1}(B) = B$, then it holds that $h(T^{-1}(h^{-1}(B))) = B,$ which means that $T^{-1}(h^{-1}(B)) = h^{-1}(B).$ If $\mu$ is ergodic, this directly implies that $\mu(h^{-1}(B)) \in \{0,1\}$, which gives us that $(h\#\mu)(B) \in \{0,1\}$, completing the proof.
\end{proof}

\begin{remark}\label{rmk:image_support}
In situations when $h:X\to Y$ is a homeomorphism onto its image, it is understood that the inverse map $h^{-1}$ is only well-defined on the set $h(X)\subseteq Y$. Moreover, if $\mu$ is $T$-invariant, then $h\#\mu\in \mathcal{P}(Y)$ can still be viewed as an invariant measure of the corresponding conjugate system $\widetilde{T}:h(X)\to h(X)$, given by $\widetilde{T} = h\circ T\circ h^{-1}$, in the sense that its restriction to the set $h(X)\supseteq \text{supp}(h\#\mu)$ is $\widetilde{T}$-invariant.
\end{remark}
We now recall the definition of the time-delay map, which provides a helpful choice of conjugating map for our purposes. For further discussion and motivation about the benefits of the time-delay map, refer to the main text.
\begin{definition}[Time-delay map]\label{def:time_delay}
    Consider a Polish space $X$, a map $T:X\to X$, an observable function $y:X\to \mathbb{R}$, and the time delay parameter $m\in\mathbb{N}$. The \textit{time-delay map} is defined as 
\begin{equation}\label{eq:delay_map}
    \Psi_{(y,T)}^{(m)}(x):=(y(x),y(T(x)),\dots, y(T^{m-1}(x)))\in \mathbb{R}^m,
    \end{equation}
    for each $x\in X$. 
\end{definition}
In \eqref{eq:delay_map}, we stress the dependence of the time-delay map $\Psi_{(y,T)}^{(m)}$ on the scalar observation function $y:X\to \mathbb{R}$, the underlying system $T:X\to X$, and the dimension $m\in\mathbb{N}$. %
When $T$ is continuous and the time-delay map $\Psi_{(y,T)}^{(m)}$ is injective, one can build a dynamical system in the reconstruction space $\mathbb{R}^m$ based on the following definition.
\begin{definition}[Delay-coordinate dynamics]\label{def:delay_dynamics}
    Consider a Polish space $X$, a map $T:X\to X$, an observable function $y:X\to \mathbb{R}$, and the time delay parameter $m\in\mathbb{N}$. Assume that $\Psi_{(y,T)}^{(m)}:X\to \mathbb{R}^m$ given in~\eqref{eq:delay_map} is injective. Then, the \textit{delay-coordinate dynamics} are given by
    \begin{equation*}\label{eq:tdd}
\hat{T}_{(y,m)}:\Psi_{(y,T)}^{(m)}(X)\to \Psi_{(y,T)}^{(m)}(X),\qquad \hat{T}_{(y,m)}:= \Big[\Psi_{(y,T)}^{(m)}\Big] \circ T \circ \Big[\Psi_{(y,T)}^{(m)}\Big]^{-1}.
    \end{equation*}
\end{definition}

Since the delay-coordinate map $\hat{T}_{(y,m)}$ is conjugate to the state-coordinate map $T:X\to X$, we now revisit Proposition~\ref{prop:conjugacy} to motivate our definition for an invariant measure in time-delay coordinates. When $T\#\mu =\mu$, the invariant measure in time-delay coordinates should be viewed as the corresponding invariant measure of the conjugate system $\hat{T}_{(y,m)}$ given by the pushforward of $\mu$ under the time-delay map~\eqref{def:time_delay}; see Definition \ref{def:delay_im} below.

\begin{definition}[Invariant Measure in Time-Delay Coordinates]\label{def:delay_im}
      Assume that the delay map $\Psi_{(y,T)}^{(m)}:X\to \mathbb{R}^m$ in Definition~\ref{def:time_delay} is injective. Then, the probability measure 
      \begin{equation}\label{eq:tdim}
          \hat{\mu}_{(y,T)}^{(m)}:=\Psi_{(y,T)}^{(m)}\# \mu \in \mathcal{P}(\mathbb{R}^m)
      \end{equation} is the corresponding \textit{invariant measure in the time-delay coordinates.} We also refer to \eqref{eq:tdim} as a \textit{delay-coordinate invariant measure.} 
\end{definition}

\subsection{Embedding Theory}\label{subsec:embedding}

This section gives the complete statements of two embedding theorems crucial for this work. In particular, we will study generalizations of the celebrated embedding theorems due to Whitney and Takens \cite{Takens1981DetectingSA,whitney1936differentiable}. Whitney showed that a generic smooth map from a smooth $n$-dimensional manifold into $\mathbb{R}^{2n+1}$ is an embedding, while Takens showed that the time-delay map from Definition \ref{def:time_delay} is generically an embedding into $\mathbb{R}^{2n+1}$ when $X = M$ is a smooth compact $n$-dimensional manifold, $T\in C^2(M,M)$ is a diffeomorphism, and $y\in C^2(M,\mathbb{R})$. 

In these statements, the phrase ``generic" indicates that the embedding property holds for an open and dense set of functions. Given that practitioners often have little control over the observable $y$, an important part of the delay map~\eqref{eq:delay_map}, the embedding property should hold for a large enough class of $y\in C^2(M,\mathbb{R})$ such that a ``randomly chosen'' observation function can be used to embed the dynamics with high probability.

Since an open and dense set in a topological space can have vanishingly small Lebesgue measure, the original statements by Takens and Whitney are insufficient in this regard \cite{hunt1992prevalence, sauer1991embedology}. In particular, one would like to say that the time-delay map constitutes an embedding for ``almost all" observation functions $y\in C^2(M,\mathbb{R})$. Along these lines, \cite{sauer1991embedology} proves generalization of both Whitney's and Takens' theorems using a measure-theoretic notion of genericity, known as prevalence, which is stronger than the topological genericity that appeared in the original statements~\cite{Takens1981DetectingSA}. We will thus develop our main results based on prevalence.

In Section~\ref{subsubsec:prevalence}, we provide an overview of the notion of prevalence, as well as some foundational results we will employ in this work. In Section~\ref{subsubsec:theorems}, we then provide complete statements of the generalized versions of Takens' and Whitneys' embedding theorems due to Sauer, Yorke, and Casdagli \cite{sauer1991embedology}. Finally, in Section \ref{subsec:borel}, we comment on the Borel-measurability of the set of embeddings, a technical detail needed to prove our main results.

\subsubsection{Prevalence}\label{subsubsec:prevalence}

The definition of prevalence generalizes the notion of ``Lebesgue-almost everywhere'' from finite-dimensional vector spaces to infinite dimensional function vector spaces~\cite{hunt1992prevalence,sauer1991embedology,ott2005prevalence}. Towards this, we first define the Lebesgue measure on an arbitrary finite-dimensional vector space. While such a definition will necessarily depend on a choice of basis for the target space, we will only be interested in whether a given set has zero or positive Lebesgue measure, and this is a property independent of one's choice of basis \cite{hunt1992prevalence, shannon2006prevalent}.

\begin{definition}[Full Lebesgue measure on finite-dimensional vector spaces]\label{def:full_leb}
Let $E$ be a  $k$-dimensional vector space and consider a basis $\{v_1,\ldots,v_k\}\subseteq E$. A set $F\subseteq E$ has full Lebesgue measure if the coefficients $\big\{(a_1,\dots, a_k) \in \mathbb{R}^k: \sum_{i=1}^{k} a_iv_i \in F\big\}$ of its basis expansion have full Lebesgue measure in $\mathbb{R}^k$. 
\end{definition}

We will next define prevalence to generalize Definition \ref{def:full_leb} to infinite dimensional vector spaces. In particular, we will define prevalence for Borel subsets of completely metrizable topological vector spaces. A topological vector space is a vector space in which the operations of vector addition and scalar multiplication are continuous functions. Moreover, the assumption of complete metrizability means that one can endow the topological vector space with a metric, for which every Cauchy sequence is convergent to a limit inside of the vector space. Additionally, this metric can always be chosen to be translation invariant. 
\begin{definition}[Prevalence]\label{def:prevalence}
    Let $V$ be a completely metrizable topological vector space. A Borel subset $S\subseteq V$ is said to be \textit{prevalent} if there is a finite-dimensional subspace $E\subseteq V$, known as a \textit{probe space}, such that for each $v\in V$, it holds that $v+e\in S$ for Lebesgue almost-every $e\in E$. 
\end{definition}
Intuitively, prevalence means that almost all perturbations of an element $v\in V$ by elements of a probe space $E$ necessarily belong to the prevalent set $S$. Note that Definition 
 \ref{def:prevalence} reduces to Definition \ref{def:full_leb} if $V$ is a finite dimensional vector space. Moreover, if $V = C^k$, the space of functions whose $k$-th order derivative exists and is continuous, then it can be shown that a prevalent set is also dense in the $C^k$-topology \cite{hunt1992prevalence}. It also holds that the finite intersection of prevalent sets remains prevalent. The proof of this fact leverages the result that any finite-dimensional subspace containing a probe space is itself a probe space. These results play crucial roles in our main theorems, so we prove them in Lemmas~\ref{lemma:contain_probe} and \ref{lemma:prevalence} for completeness.
\begin{lemma}\label{lemma:contain_probe}
    Let $V$ be a completely metrizable topological vector space, assume $S\subseteq V$ is prevalent, and assume that $E\subseteq V$ is a probe space. If $E'\subseteq V$ is a finite-dimensional vector space with $E\subseteq E'$, then $E'$ is also a probe space.
\end{lemma}
\begin{proof}
We will denote $\text{dim}(E) = k$  and $\text{dim}(E') = \ell \geq k$. Let $e_1,\dots, e_{\ell}\in E'$ be basis vectors for $E'$, where $e_1,\dots, e_k\in E$ are basis vectors for $E$. By the assumption that $E$ is a probe space, we know that for all $v\in V$ it holds that $v + \sum_{i=1}^k a_ie_i \in S$ for Lebesgue almost all $(a_1,\dots, a_k)\in \mathbb{R}^k$. Now, let $v\in V$ be fixed and notice that 
  \begin{equation}\label{eq:decompose}
      v+ \sum_{i=1}^{\ell}a_i e_i = \Bigg(v+ \sum_{i = k+1}^{\ell}a_ie_i\Bigg) + \sum_{i=1}^ka_i e_i.
  \end{equation}
  Due to the decomposition \eqref{eq:decompose} and the fact that $E$ is a probe space, it holds for fixed $v\in V$ and $a_{k+1},\dots, a_{\ell}\in \mathbb{R}$ that $v+\sum_{i=1}^{\ell}a_ie_i \in S$ for $\lambda_k$-almost every $(a_1,\dots, a_k)\in \mathbb{R}^k$. 
  
To conclude the proof, it remains to show that $v+\sum_{i=1}^{\ell} a_i e_i \in S$ for $\lambda_{\ell}$-almost every $(a_1,\dots, a_{\ell})\in\mathbb{R}^{\ell}$, which will be a consequence of the Tonelli's theorem~\cite[Theorem.~3.2]{stein2009real}. Towards this, we define the set $B:=\{(a_1,\dots, a_{\ell}): v+\sum_{i=1}^{\ell}a_ie_i \notin S\}\subseteq \mathbb{R}^{\ell}$, and will show $\lambda_{\ell}(B) = 0.$ To apply Tonelli's theorem to  $\chi_B$, the characteristic function of the set $B$, we must first justify that $B$ is Borel measurable. Towards this, define $f(a_1,\dots,a_{\ell}):= v+\sum_{i=1}^{\ell}a_i e_i$,  a continuous function from $\mathbb{R}^{\ell}$ to $V$. 
  Since $S\subseteq V$ is Borel, it holds that $B= f^{-1}(S)$ is Borel as well. 
  
  For notational convenience, we will now rewrite $\chi_B$ as a function $\phi: \mathbb{R}^{k}\times \mathbb{R}^{\ell - k}\to \mathbb{R}$, given by 
  $$\phi(x,y):= \chi_B(x_1,\dots, x_k, y_1,\dots, y_{\ell - k}),\qquad x\in\mathbb{R}^k,\qquad y\in \mathbb{R}^{\ell-k}.$$ Note that for every fixed $y \in \mathbb{R}^{\ell - k}$, $\phi(x,y) = 0$ for $\lambda_{k}$-almost every $x\in \mathbb{R}^k$. Then, it follows by Tonelli's theorem that
  \begin{align*}
      \lambda_{\ell}(B) = \int_{\mathbb{R}^{\ell}}\chi_Bd\lambda_{\ell} = \int_{\mathbb{R}^{\ell -k}}\bigg(\underbrace{ \int_{\mathbb{R}^{k}}  \phi(x,y) d\lambda_{k}(x) }_{ = 0}\bigg)d\lambda_{\ell - k}(y) = 0,
  \end{align*}
  which concludes the proof.
\end{proof}
It is now a direct consequence of Lemma \ref{lemma:contain_probe} that the \text{finite} intersection of prevalent sets is prevalent.
\begin{lemma}\label{lemma:prevalence}
Let $S_1$ and $S_2$ be two prevalent subsets of a completely metrizable topological vector space $V$. Then, the intersection $S_1\cap S_2$ is also prevalent.
\end{lemma}
\begin{proof}
Let $E_1 = \text{span}\{v_1,\dots, v_k\}$ be a probe space for $S_1$ and  $E_2 = \text{span}\{w_1,\dots, w_{\ell}\}$ be a probe space for $S_2$. 
Define $E':=\text{span}(\{v_1,\dots, v_k, w_1,\dots, w_{\ell}\})$, and consequently $E_1,E_2\subseteq E'.$ Thus, by Lemma~\ref{lemma:contain_probe}, $E'$ is a probe space for both $S_1$ and $S_2$. The fact that $E'$ is a probe space for $S_1\cap S_2$ implies that $S_1\cap S_2$ is prevalent.
\end{proof}

\subsubsection{Generalized Whitney's and Takens' Embedding Theorems}\label{subsubsec:theorems}
In addition to using the theory of prevalence, the results in \cite{sauer1991embedology} generalize the classical Takens and Whitney theorems in another crucial way. It is often the case that dynamical trajectories are asymptotic to a compact attracting set $A$, which has a fractal structure and is not a manifold. The classical Takens and Whitney theorems require that such an attractor is contained within a smooth, compact manifold of dimension $d$ to embed the dynamics in $2d+1$-dimensional reconstruction space. The fractal dimension $d_A$ of the set $A$ might be much less than the manifold dimension, i.e., $d_A \ll d$. In such cases, it is desirable to consider more efficient approaches that can guarantee a system reconstruction in $2d_A+1$-dimensions. Towards this, we now recall the following definition of box-counting dimension, which will serve as our notion for the dimension of a fractal set.
 
 \begin{definition}[Box counting dimension]\label{def:box}
     Let $A\subseteq \mathbb{R}^n$ be a compact set. Its \textit{box counting dimension} is $$\text{boxdim}(A):=\displaystyle\lim_{\varepsilon \to 0} \frac{\log N(\varepsilon)}{\log(1/\varepsilon)},$$ where $N(\varepsilon)$ is the number of boxes with side-length $\varepsilon$ required to cover $A$. When the limit does not exist, one can define the upper and lower box-counting dimensions by replacing the limit with liminf and limsup, respectively.
 \end{definition}
 The following generalization of Whitney's embedding theorem comes from~\cite[Theorem 2.3]{sauer1991embedology}.
\begin{theorem}[Fractal Whitney Embedding]\label{thm:whitney}
    Let $A\subseteq \mathbb{R}^n$ be compact,  $d:= \textup{boxdim}(A)$, and $m > 2d$ be an integer. Then, almost every $F\in C^1(\mathbb{R}^n,\mathbb{R}^m)$ is injective on $A$.
\end{theorem}
\begin{remark}\label{rmk:linear}
    The probe space for the prevalent set appearing in Theorem \ref{thm:whitney} can be chosen as the $nm$-dimensional space of linear maps between $\mathbb{R}^n$ and $\mathbb{R}^m.$
\end{remark}
There are certain assumptions required on the periodic points of a dynamical system for Takens' theorem to hold. Assumption \ref{assumption:1} makes these technical assumptions concise and easy to reference in our main results. In what follows, $DT^p$ denotes the derivative of the $p$-fold composition map $T^p$.
\begin{assumption}[Technical assumption on periodic points]\label{assumption:1}
  Let $T:U\to U$ be a diffeomorphism of an open set $U\subseteq \mathbb{R}^n$, $A\subseteq U$ be compact, and $m\in\mathbb{N}$. For each $p\leq m$:
  \begin{enumerate}
      \item The set $A_p\subseteq A$ of $p$-periodic points satisfies $\textup{boxdim}(A_p) < p/2$.
      \item The linearization $DT^p$ of each periodic orbit has distinct eigenvalues. %
  \end{enumerate}
  \end{assumption}

  We remark that when the diffeomorphism $T$ is given by the time-$\tau$ flow map of a Lipschitz continuous vector field, one can choose $\tau$ sufficiently small such that Assumption \ref{assumption:1} is satisfied; see \cite{sauer1991embedology}. The following generalization of Takens' theorem can be found in \cite[Theorem 2.7]{sauer1991embedology}.
\begin{theorem}[Fractal Takens' Embedding]\label{thm:takens}
    Let $T:U \to U$ be a diffeomorphism of an open set $U\subseteq \mathbb{R}^n$, $A\subseteq U$ be compact and  $m > 2d$ where $d=\textup{boxdim}(A)$. Assume that the periodic points of $T$ with degree at most $m$ satisfy Assumption~\ref{assumption:1} on $A$. Then, it holds that $\Psi_{(y,T)}^{(m)}$ is injective on $A$, for almost all $y\in C^1(U,\mathbb{R}).$ 
\end{theorem}
\begin{remark}
    The probe space for the prevalent set appearing in Theorem \ref{thm:takens} can be chosen as the space of polynomials for $n$ variables of degree at most $2m$.
\end{remark}
 We note that when the box-counting dimension does not exist, the generalized version of the Takens' and Whitney's embedding theorems (see Theorems \ref{thm:whitney} and \ref{thm:takens}) hold in terms of the lower box-counting dimension (see Definition \ref{def:box}). Moreover, the statements of Theorems \ref{thm:whitney} and \ref{thm:takens} in~\cite{sauer1991embedology} indicate that $F$ and $\Psi_{(y,T)}^{(m)}$ are also immersions on each compact subset of a smooth manifold contained in $A$, in addition to being injections. We do not need this property for our proofs, so we have left it out of the statements of Theorems~\ref{thm:whitney} and~\ref{thm:takens} for simplicity.

\subsubsection{On the Borel-Measurability of the Set of Embeddings}\label{subsec:borel}
It is important to understand how Theorems \ref{thm:whitney} and \ref{thm:takens} can be precisely understood within the theory of prevalence. In particular, Definition \ref{def:prevalence} requires that a prevalent set is a Borel subset of completely metrizable topological vector space. While \cite{sauer1991embedology} does not discuss the measurability of the sets of functions appearing in Theorems \ref{thm:whitney} and \ref{thm:takens}, this property is essential for the proofs of our main results. We leverage Lemma \ref{lemma:prevalence} to construct new prevalent sets via the intersection operation. This construction is made possible by Tonelli's theorem, which we can apply when the prevalent sets in question are Borel measurable (see our proof of Lemma~\ref{lemma:prevalence}). 

In this section, we comment on the structure of $C^1(U,\mathbb{R}^k)$ as a completely metrizable topological vector space, and we explicitly write down an appropriate translation-invariant metric on $C^1(U,\mathbb{R}^k)$. We will then move on to verify that the sets of functions appearing in the statements of Theorems~\ref{thm:whitney} and \ref{thm:takens}, for which the conclusions of injectivity are satisfied, are indeed Borel sets in this topology.

\begin{remark}\label{rmk:metric}
    For any open subset $U\subseteq \mathbb{R}^n$, we remark that $C^r(U,\mathbb{R}^k)$ is a completely metrizable topological vector space.  The construction we provide here follows \cite{Alazard2024}. In particular, consider a sequence of compact sets $K_j \subseteq U$, such that $K_j\subseteq \text{int}(K_{j+1})$ and $\bigcup_{j=1}^{\infty}K_j = U$.  Then define the metric $\textrm{d}_{C^r}:C^r(U,\mathbb{R}^k)\times C^r(U,\mathbb{R}^k) \to [0,\infty)$ by setting
    $$\textrm{d}_{C^r}(f,g):=  \sum_{j=1}^{\infty}\frac{1}{2^j} \frac{\textrm{d}_{j}(f,g)}{1+\textrm{d}_{j}(f,g)},\qquad \textrm{d}_{j}(f,g):=\max_{|\alpha|\leq r}\sup_{x\in K_j}\|\partial^{\alpha} f(x) -\partial^{\alpha} g(x)\|_2,\qquad f,g\in C^r(U,\mathbb{R}^k),$$
    where $\|\cdot\|_2$ denotes the Euclidean $2$-norm and $\alpha\in\mathbb{N}^{n}$ is a multi-index. Then, $C^r(U,\mathbb{R}^k)$ is complete with respect to the metric $\textrm{d}_{C^r},$ which is translation-invariant and continuous with respect to both scalar multiplication and vector addition \cite[Proposition 1.25]{Alazard2024}. 
\end{remark}
We will next verify that the set of smooth maps that are injective on a compact set $A\subseteq \mathbb{R}^n$ is indeed a Borel set with respect to the metric-topology induced by $\textrm{d}_{C^1}$; see Remark \ref{rmk:metric} .
\begin{proposition}\label{prop:borel1}
    Let $U\subseteq \mathbb{R}^n$ be an open set,  $A\subseteq U$ be compact, and $k\in\mathbb{N}$. Then, the set $\{f\in C^1(U,\mathbb{R}^k): f\textup{ is injective on }A\}$ is a Borel subset of $C^1(U,\mathbb{R}^k)$.
\end{proposition}
\begin{proof}
    For $\varepsilon > 0$, define the set $D_{\varepsilon}:=\{(x,y)\in U\times U: \|x-y\|_2 < \varepsilon\}$ and note that $D_{\varepsilon}$ is open. To see this, observe that $d:U\times U \to \mathbb{R}$, given by $d(x,y):= \|x-y\|_2$ is continuous and that $D_{\varepsilon} = d^{-1}((-\infty,\varepsilon))$ is the preimage of an open set under a continuous function, and is hence open. Next, consider the set 
    $$\bm{B}_{\varepsilon}:=\{f\in C^1(U,\mathbb{R}^k): \|f(x) - f(y)\|_2 > 0 ,\, (x,y)\in (A\times A)\setminus D_{\varepsilon}\}.$$
    We claim that $\bm{B}_{\varepsilon}$ is an open set. To see this, we will fix $f\in \bm{B}_{\varepsilon}$ and produce an open ball of radius $\gamma > 0$, which is centered at $f$ and fully contained within $\bm{B}_{\varepsilon}$.  Towards this, let $K_j\subseteq U$ be a sequence of compact sets with $K_j\subseteq \text{int}(K_{j+1})$ and $\bigcup_{j\in\mathbb{N}}K_j = U$. Now, let $j\in \mathbb{N}$ be chosen such that $K_j\subseteq U$ is a compact set satisfying $A\subseteq K_j$. Furthermore, note that since the map $(x,y)\mapsto \|f(x)-f(y)\|_2$ is continuous over the compact set $(A\times A)\setminus D_{\varepsilon}$, there must exist some $a \in (0,1)$, such that  $\|f(x)-f(y)\|_2 > a >0,$ for all $(x,y) \in (A\times A)\setminus D_{\varepsilon}$. Now, choose $\gamma:= a/2^{j+1} > 0$. Then, for any $g\in C^1(U,\mathbb{R}^k)$ satisfying $\textrm{d}_{C^1}(f,g) < \gamma$, following the definition of $\textrm{d}_{C^1}$ (see Remark \ref{rmk:metric}), we have that for any $x\in A$ that
    $$\frac{1}{2^j}\frac{\|f(x)-g(x)\|_2}{1+ \|f(x)-g(x)\|_2} \leq \mathrm{d}_{C^1}(f,g) < \gamma,$$
    which rearranges to yield $$\|f(x) - g(x)\|_2 < \frac{\gamma 2^j}{1-\gamma 2^j} = \frac{a/2}{1-a/2}< \frac{a}{2}.$$
    Thus, for all $(x,y)\in (A\times A)\setminus D_{\varepsilon}$, it follows that \begin{align*}
      a &<   \|f(x) - f(y) \|_2\\ &\leq \|f(x) - g(x)\|_2 + \|f(y) - g(y)\|_2 +\|g(x) - g(y)\|_2 \\
        & < \frac{a}{2} + \frac{a}{2} + \|g(x) - g(y)\|_2,
   \end{align*} 
   which implies $\|g(x) - g(y)\|_2 > 0.$ By definition, we have $g\in \bm{B}_{\varepsilon}$. Since $g$ was arbitrary, we have $\{g\in C^1(U,\mathbb{R}^k): \textrm{d}_{C^1}(f,g) < \gamma\} \subset \bm{B}_{\varepsilon}$. Since  $f\in\mathbf{B}_{\varepsilon}$ was arbitrary,  we verified that $\bm{B}_{\varepsilon}$ is open. 
   
   Next, we claim that 
\begin{equation}\label{eq:set_inc}
       \{ f\in C^1(U,\mathbb{R}^k): f \text{ is injective on } A\} =  \bigcap_{n\in \mathbb{N}} \bm{B}_{1/n}.
   \end{equation}
   If we can establish \eqref{eq:set_inc}, then we have written the set of injective functions as a countable intersection of open sets, which will conclude our proof. We will  verify that \eqref{eq:set_inc} holds by showing both set inclusions.
   
 We first prove the ``$\subseteq $" direction. Assume that $f\in C^1(U,\mathbb{R}^k)$ is injective on $A$. Defining $D:=\{(x,x):x\in A\}$, it holds by definition of injectivity that $\|f(x)-f(y)\|_2 > 0$, for all $(x,y) \in (A\times A)\setminus D$. In particular, since $D \subseteq D_{\varepsilon}$ for all $\varepsilon > 0$ it holds that $\|f(x) - f(y)\|_2 > 0$, for all $(x,y)\in (A\times A)\setminus D_{\varepsilon}$, for any $\varepsilon >0$, and thus, $f\in \bm{B}_{1/n}$, for all $n\in\mathbb{N}$, i.e., $f\in \bigcap_{n\in \mathbb{N}}\bm{B}_{1/n}.$
 
 We now prove the ``$\supseteq $" direction. Assume that $f\in C^1(U,\mathbb{R}^k)$ satisfies $\|f(x) - f(y)\|_2 > 0$ for all $(x,y)\in (A\times A)\setminus D_{1/n}$, for each $n\in\mathbb{N}$. Let $(x_0,y_0) \in A\times A$ be fixed and $x_0 \neq y_0$. If we choose $N_0\in\mathbb{N}$ to satisfy $N_0>1/\|x_0-y_0\|_2$, then it holds that $(x_0,y_0)\in (A\times A)\setminus D_{1/N_0}$ and therefore $\|f(x_0) - f(y_0)\|_2 > 0$, which means $f(x_0)\neq f(y_0).$ Since $(x_0,y_0)\in A\times A$ where $x_0\neq y_0$ was chosen arbitrarily, we have by definition that $f$ is injective on $A$, which concludes the proof.
\end{proof}
We next verify that the set of observation functions $y\in C^1(U,\mathbb{R})$, such that the delay map $\Psi_{(y,T)}^{(k)}$ is injective on a compact set $A\subseteq U$ is also Borel in the metric-topology induced by $\textrm{d}_{C^1}.$ Throughout our proof, we use the fact that the convergence $\lim_{n\to\infty}\textrm{d}_{C^r}(f_n,g)= 0$ is equivalent to the convergence $\lim_{n\to\infty}\textrm{d}_j(f_n,g)= 0$, for all $j\in\mathbb{N}$; see \cite[Proposition 1.23]{Alazard2024}.
\begin{proposition}\label{prop:borel2}
    Let $U\subseteq \mathbb{R}^n$ be open, $A\subseteq U$ be compact, $T\in C^1(U,U)$ be a diffeomorphism, and $k\in \mathbb{N}$. Then, the set $\{y\in C^1(U,\mathbb{R}): \Psi_{(y,T)}^{(k)} \text{ is injective on }A\}$ is Borel.
\end{proposition}
\begin{proof}
   Define the operator $\bm{\Psi}:C^1(U,\mathbb{R})\to C^1(U,\mathbb{R}^k)$ as $\bm{\Psi}(y)=\Psi_{(y,T)}^{(k)} = (y,y\circ T,\dots, y\circ T^{k-1}) $ and observe that
    $$\{y\in C^1(U,\mathbb{R}): \Psi_{(y,T)}^{(k)} \text{ is injective on }A\} = \bm{\Psi}^{-1}(\{ F\in C^1(U,\mathbb{R}^k): F \text{ is injective on }A\}).$$
    Since $\{ F\in C^1(U,\mathbb{R}^k): F \text{ is injective on }A\}$ is Borel by Proposition \ref{prop:borel1}, it suffices to show that $\bm{\Psi}$ is continuous, hence Borel measurable. Towards this, let us assume that $\lim_{\ell\to\infty}\textrm{d}_{C^1}(y_{\ell},y)=0$, which equivalently means $\lim_{\ell\to\infty}\textrm{d}_j(y_{\ell},y)=0$, for all $j\in\mathbb{N}$. Recall $\textrm{d}_j$ is defined in Remark \ref{rmk:metric}, based upon a sequence of compact sets $K_j\subseteq U$, such that $K_j\subseteq \text{int}(K_{j+1})$ and $\bigcup_{j\in\mathbb{N}}K_j = U$. Our goal is to show that $\lim_{\ell\to\infty}\textrm{d}_{C^1}(\bm{\Psi}(y_{\ell}),\bm{\Psi}(y)) = 0$, as well.
    
    First, let the multi-index $\alpha \in \mathbb{N}^n$ satisfy $|\alpha| = 0$, so $\partial^\alpha \bm{\Psi}(y) = \bm{\Psi}(y)$. We fix $j\in \mathbb{N}$, and for each $0\leq i \leq k-1$ write $C_i:=T^i(K_j)$, which is compact. Next, we select $j'\in \mathbb{N}$, such that $C_i \subseteq K_{j'}$, for all $0\leq i \leq k-1$. Then, it holds that
    \begin{align*}
        \sup_{x\in K_j}\| \bm{\Psi}(y_{\ell})(x) - \bm{\Psi}(y)(x)\|^2_2 = \sup_{x\in K_j}\sum_{i=0}^{k-1} ( y_{\ell}(T^i(x)) - y(T^i(x)) )^2 &\leq \sum_{i=0}^{k-1} \sup_{x\in K_j} (y_{\ell}(T^i(x))-y(T^i(x)))^2 \\
        &= \sum_{i=0}^{k-1} \sup_{x\in C_i} (y_{\ell}(x)-y(x))^2\\
        &\leq \sum_{i=0}^{k-1} \sup_{x\in K_j'} (y_{\ell}(x)-y(x))^2 \\
         & \leq k \textrm{d}_{j'}(y_{\ell},y)^2 \\
        &\xrightarrow[]{\ell\to\infty} 0.
    \end{align*}
   Now, we consider the case when $\alpha\in\mathbb{N}^n$ satisfies $|\alpha | = 1.$ In this case, we have that 
   
    \begin{align*}
        \sup_{x\in K_j}\|\partial^{\alpha} \bm{\Psi}(y_{\ell})(x) - \partial^{\alpha} \bm{\Psi}(y)(x)\|^2_2 &= \sup_{x\in K_j}\sum_{i=0}^{k-1}\Big( \partial^{\alpha} y_{\ell}(T^i(x)) - \partial^{\alpha}y(T^i(x))\Big)^2 \\
        & = \sup_{x\in K_j}\sum_{i=0}^{k-1}\Big( (\nabla y_{\ell}(T^i(x)) - \nabla y(T^i(x)))\cdot \partial^{\alpha} T^i(x)\Big)^2  \\
        &\leq \sup_{x\in K_j}\sum_{i=0}^{k-1}\|\nabla y_{\ell}(T^i(x)) - \nabla y(T^i(x))\|_2^2 \cdot \| \partial^{\alpha} T^i(x)\|_2^2 \\
        &\leq \sum_{i=0}^{k-1}\sup_{x\in K_{j'}}\|\nabla y_{\ell}(x) - \nabla y(x)\|_2^2 \cdot\sup_{x\in K_j} \| \partial^{\alpha} T^i(x)\|_2^2\\
        &\xrightarrow[]{\ell \to \infty} 0,
    \end{align*}
    where we have applied the Cauchy-Schwartz inequality and again used the fact that $C_i\subseteq K_{j'}$ for each $0\leq i \leq k-1$. Thus, we have shown that 
    $$\lim_{\ell \to \infty}\textrm{d}_j(\bm{\Psi}(y_{\ell}),\bm{\Psi}(y)) = \lim_{\ell\to \infty}\max_{|\alpha|\leq 1} \sup_{x\in K_j}\|\partial^{\alpha} \bm{\Psi}(y_{\ell})(x) - \partial^{\alpha} \bm{\Psi}(y)(x)\|_2= 0, \qquad \forall j\in\mathbb{N}.$$
    Equivalently, $\textrm{d}_{C^1}(\bm{\Psi}(y_{\ell}),\bm{\Psi}(y))\to 0$ as $\ell\to \infty$, which completes the proof. 
\end{proof}

\section{Main Theoretical Results}\label{sec:results}
\setcounter{theorem}{0}
\setcounter{proposition}{0}
\setcounter{theorem}{0}
\setcounter{lemma}{0}
\setcounter{remark}{0}
\setcounter{definition}{0}
\setcounter{corollary}{0}

In this section, we state and prove our main results. The statements of our two main theorems appear in Section \ref{subsec:statements}, discussions follow in Section \ref{subsec:discussions}, and the complete proofs can be found in Section \ref{subsec:proofs}.

\subsection{Statements}\label{subsec:statements}
We are now ready to state our main results. In the following, we use the notation $d_{\mu}:=\text{boxdim}(\text{supp}(\mu))$ to denote the box-counting dimension of the support of a probability measure $\mu\in \mathcal{P}(\mathbb{R}^k)$. Our first result, Theorem \ref{thm:1}, which appears as Theorem 1 in the main text, shows that equality of delay-coordinate invariant measures implies topological conjugacy of the underlying dynamics. 
\begin{theorem}\label{thm:1}
    Let $T,S:U\to U$ be diffeomorphisms of an open set $U\subseteq \mathbb{R}^{n}$ and suppose that $\mu\in \mathcal{P}(U)$ is $T$-invariant, $\nu \in \mathcal{P}(U)$ is $S$-invariant, and $\textnormal{supp}(\mu) ,\textnormal{supp}(\nu) \subseteq U$ are compact. Further assume that $m > 2\max\{ d_{\mu},d_{\nu}\}$, and that the periodic points of $T$ and $S$ with degree at most $m$ satisfy Assumption~\ref{assumption:1} on $\textnormal{supp}(\mu)$ and $\textnormal{supp}(\nu)$, respectively. Then, equality $\hat{\mu}_{(y,T)}^{(m+1)} = \hat{\nu}_{(y,S)}^{(m+1)}$ of the delay-coordinate invariant measures implies topological conjugacy of the maps $T|_{\textup{supp}(\mu)}$ and $S|_{\textup{supp}(\nu)}$, for almost every $y\in C^1(U,\mathbb{R})$. 
\end{theorem}
\begin{remark}
    Our proof of Theorem \ref{thm:1} primarily relies on the topology of delay coordinates, and the conclusion of Theorem \ref{thm:1} still follows if one replaces the assumption $\hat{\mu}_{(y,T)}^{(m+1)}=\hat{\nu}_{(y,S)}^{(m+1)}$ with the weaker statement that  $\text{supp}(\hat{\mu}_{(y,T)}^{(m+1)})=\text{supp}(\hat{\nu}_{(y,S)}^{(m+1)})$.
\end{remark}
As a Corollary to Theorem \ref{thm:1}, we also demonstrate that ergodicity is preserved under the equality of delay-coordinate invariant measures. 
\begin{corollary}\label{cor:1S}
     Let $T,S:U\to U$ be diffeomorphisms of an open set $U\subseteq \mathbb{R}^{n}$ and suppose that $\mu\in \mathcal{P}(U)$ is $T$-invariant, $\nu \in \mathcal{P}(U)$ is $S$-invariant, and $\textnormal{supp}(\mu) ,\textnormal{supp}(\nu) \subseteq U$ are compact. Further assume that $m > 2\max\{ d_{\mu},d_{\nu}\}$, and that the periodic points of $T$ and $S$ with degree at most $m$ satisfy Assumption~\ref{assumption:1} on $\textnormal{supp}(\mu)$ and $\textnormal{supp}(\nu)$, respectively. If $\mu$ is $T$-ergodic, then equality $\hat{\mu}_{(y,T)}^{(m+1)} = \hat{\nu}_{(y,S)}^{(m+1)}$ of the delay-coordinate invariant measures implies that $\nu$ is also $S$-ergodic, for almost every $y\in C^1(U,\mathbb{R})$.
\end{corollary}
Next, we consider the case when we have equality of a finite number of delay-coordinate invariant measures corresponding to different observation functions. In Theorem \ref{thm:2}, which appears in the main text as Theorem 2, we prove that equality of such measures, combined with a certain initial condition, is sufficient information to uniquely determine the full dynamics on the support of an invariant measure.
\begin{theorem}\label{thm:2}
Let $T,S:U\to U$ be diffeomorphisms of an open set $U\subseteq \mathbb{R}^{n}$, suppose that $\mu\in \mathcal{P}(U)$ is both $T$-invariant and $S$-invariant, and assume that $\textnormal{supp}(\mu) \subseteq U$ is compact. Further assume that $m > 2d_{\mu}$ and that the periodic points of $T$ and $S$ with degree at most $m$ satisfy Assumption \ref{assumption:1} on $\textnormal{supp}(\mu)$. If \begin{enumerate}
     \item[(1)] there exists $x^*\in B_{\mu,T}\cap \textup{supp}(\mu)$, such that $T^k(x^*) = S^k(x^*)$ for $1\leq k \leq m-1$, and 
     \item[(2)] $\hat{\mu}_{(y_j,T)}^{(m+1)} = \hat{\mu}_{(y_j,S)}^{(m+1)}$ for $1\leq j \leq m$, where $Y:=(y_1,\dots, y_m)\in C^1(U,\mathbb{R}^m)$ is a vector-valued observable,
 \end{enumerate}
 then for almost every $Y\in C^1(U,\mathbb{R}^m)$, it holds that  that $T = S$ everywhere on $\textup{supp}(\mu)$.
\end{theorem}
Lastly, we introduce an alternative measure-based approach for performing unique system identification when one can observe the full state. These details are summarized in Proposition \ref{prop:graph_fullS}.

\begin{proposition}\label{prop:graph_fullS}
    Let $T,S:U\to U$ be diffeomorphisms of an open set $U\subseteq \mathbb{R}^{n}$, suppose that $\mu\in \mathcal{P}(U)$ is both $T$-invariant and $S$-invariant, and assume that $\textnormal{supp}(\mu) \subseteq U$ is compact. Then, the equality $(Y,Y\circ T)\# \mu = (Y,Y\circ S)\# \mu$ implies $T = S$ on $\textup{supp}(\mu)$, for almost all $Y\in C^1(U,\mathbb{R}^m)$, where $m > 2d_{\mu}$.
\end{proposition}
\subsection{Discussion}\label{subsec:discussions}
Theorems~\ref{thm:1} and~\ref{thm:2} are our main results for comparing dynamical systems through delay-coordinate invariant measures, based upon uniform time-delay embeddings. We remark that our results hold in great generality, as we place no assumptions on the invariant measures under consideration. In particular, these statements still hold when the invariant measures are singular with respect to the Lebesgue measure and have fractal support, a common situation for attracting dynamical systems~\cite{grassberger1983measuring}. Moreover, we have formulated our results using the mathematical theory of prevalence~\cite{hunt1992prevalence}, arguing that the conclusions of Theorems \ref{thm:1} and \ref{thm:2} hold for almost all observation functions. Practitioners often have little control over the measurement device used for data collection, so our proposed approach of comparing dynamical systems using invariant measures in time-delay coordinates remains broadly applicable.

Notably, Theorem \ref{thm:2} and Proposition~\ref{prop:graph_fullS} motivate two different strategies for uniquely recovering the underlying dynamical system using measure-based comparisons. In particular, we can choose to study either 
$$\underbrace{m\text{ measures in }\mathcal{P}(\mathbb{R}^{m+1})}_{\text{Theorem \ref{thm:2} + initial condition}} \, \text{ or } \, \underbrace{\text{one measure in }\mathcal{P}(\mathbb{R}^{2m})}_{\text{Proposition \ref{prop:graph_fullS}}}  $$
to recover the dynamics. Given that evaluating metrics on the space of probability measures suffers from the curse of dimensionality~\cite{panaretos2019statistical}, the former approach based upon Theorem \ref{thm:2} and the delay-coordinate invariant measure is likely to be more computationally feasible. 

\subsection{Proof of Main Results}\label{subsec:proofs}
We are now ready to present the proofs of our main results.  The proof of Theorem \ref{thm:1}, which is a consequence of the generalized Takens embedding theorem (Theorem \ref{thm:takens}) combined with an observation about the topology of time-delay coordinate systems, appears in Section \ref{subsec:proofs1}. In Section \ref{subsec:proofs2}, we then leverage Theorem~\ref{thm:1}, the generalized Takens embedding theorem (Theorem \ref{thm:takens}), and the generalized Whitney embedding theorem (Theorem~\ref{thm:whitney}), to prove Theorem \ref{thm:2}. The proof of Proposition \ref{prop:graph_fullS} then appears in Section \ref{subsec:proofs3}.
\subsubsection{Proof of Theorem \ref{thm:1}}\label{subsec:proofs1}
In this section, we present a complete proof of Theorem \ref{thm:1}. We begin by establishing a useful lemma that relates the support of a probability measure to the support of its pushforward under a continuous mapping.
\begin{lemma}\label{lemma:1}
    Let $X$ and $Y$ be Polish spaces and $\mu\in \mathcal{P}(X)$ be a Borel probability measure with compact support, and assume that $f:X\to Y$ is continuous. Then, $f(\textup{supp}(\mu)) = \textup{supp}(f\# \mu).$ 
\end{lemma}
\begin{proof}
We prove the result via double inclusion, beginning first with the ``$\supseteq $" direction. Note that since $f$ is continuous and $\text{supp}(\mu)$ is compact in $X$, it holds that $D:=f(\text{supp}(\mu))$ is compact in $Y$. Moreover, using the fact that $\text{supp}(\mu) \subseteq f^{-1}(D)$, we obtain that 
    $$(f\# \mu)(D) = \mu(f^{-1}(D)) \geq \mu(\text{supp}(\mu)) = 1.$$
     Since $D$ is closed in $Y$, this implies that $\text{supp}(f\# \mu)\subseteq D = f(\text{supp}(\mu))$. 
     
     To prove the ``$\subseteq$" inclusion, we note by continuity that $C:=f^{-1}(\text{supp}(f\# \mu))$ is a closed subset of $X$ which satisfies
$$\mu(C) = \mu(f^{-1}(\text{supp}(f\# \mu))) = (f\# \mu)(\text{supp}(f\# \mu)) = 1.$$
Thus, $\text{supp}(\mu)\subseteq C = f^{-1}(\text{supp}(f\#\mu)),$ which implies that 
\begin{equation*}\label{eq:supports}
   f(\text{supp}(\mu)) \subseteq f(f^{-1}(\text{supp}(f\# \mu)) \subseteq \text{supp}(f\#\mu). 
\end{equation*}
This establishes the inclusion $f(\text{supp}(\mu)) \subseteq \text{supp}(f\# \mu)$ and completes the proof. 
\end{proof}
The conclusion of Theorem \ref{thm:takens} states that the time-delay map $\Psi_{(y,T)}^{(m)}:U\to \mathbb{R}^m$ is injective on a compact subset $A\subseteq \mathbb{R}^n$. The following lemma will allow us to conclude that $\Psi_{(y,T)}^{(m)}:A\to \Psi_{(y,T)}^{(m)}(A)$ is in-fact a homeomorphism. This fact will be needed when we construct the conjugating map (see Definition \ref{def:conjugacy}) appearing in the conclusion of Theorem \ref{thm:1}.
\begin{lemma}{(\cite[Proposition~13.26]{sutherland2009introduction})}\label{lemma:2}
    Let $X$ and $Y$ be compact spaces and assume that $f:X\to Y$ is continuous and invertible. Then, $f$ is a homeomorphism, i.e., $f^{-1}$ is also continuous.
\end{lemma}

It is a well-known fact that the graph $\mathcal{G}[T]:=\{(x,T(x)):x\in X\}$ of a continuous function $T:X\to X$ uniquely characterizes $T$ at all points in $X$. Indeed, if $\mathcal{G}[T] = \mathcal{G}[S]$ for some continuous $S:X\to X$, then one can easily deduce that $T = S.$ Our proof of Theorem \ref{thm:1} can be viewed as a generalization of this observation to the case when the graph $\mathcal{G}[T]$ is distorted by the time-delay map; see Definition \ref{def:time_delay}.
\begin{proof}[Proof of Theorem \ref{thm:1}]
 By Theorem \ref{thm:takens}, we have that
 \begin{align*}
     \mathcal{Y}_1&:=\{y\in C^1(U,\mathbb{R}):\Psi_{(y,S)}^{(m)}\text{ is injective on }\text{supp}(\nu)\}, \\
      \mathcal{Y}_2&:=\{y\in C^1(U,\mathbb{R}):\Psi_{(y,T)}^{(m)}\text{ is injective on }\text{supp}(\mu)\},
 \end{align*}
are prevalent subsets of $C^1(U,\mathbb{R})$. By Lemma~\ref{lemma:prevalence}, it holds that the intersection $\mathcal{Y}:=\mathcal{Y}_1\cap \mathcal{Y}_2$ is also prevalent in $C^1(U,\mathbb{R})$. Now, let $y\in \mathcal{Y}$ be fixed and assume that 
 \begin{equation}\label{eq:assumption}
     \Psi_{(y,S)}^{(m+1)}\# \nu = \Psi_{(y,T)}^{(m+1)}\#\mu.
 \end{equation}
    Since the mappings $$\Psi_{(y,S)}^{(m)}\Big|_{\text{supp}(\nu)}:\text{supp}(\nu)\to \Psi_{(y,S)}^{(m)}(\text{supp}(\nu)),\qquad \Psi_{(y,T)}^{(m)}\Big|_{\text{supp}(\mu)}:\text{supp}(\mu)\to \Psi_{(y,T)}^{(m)}(\text{supp}(\mu))$$ 
    are continuous, invertible, and the sets $\text{supp}(\nu)$ and $\text{supp}(\mu)$ are compact, it follows from Lemma~\ref{lemma:2} that the map $  \Theta_y:\text{supp}(\nu) \to \text{supp}(\mu)$, given by
    \begin{equation}\label{eq:theta}
    \Theta_y(x) :=\Bigg( \Big[\Psi_{(y,T)}^{(m)}\Big]\bigg|_{\text{supp}(\mu)}^{-1} \circ \Big[\Psi_{(y,S)}^{(m)}\Big]\bigg|_{\text{supp}(\nu)}\Bigg)(x),\qquad \forall x\in \text{supp}(\nu),
    \end{equation}
    is a well-defined homeomorphism. We now aim to show that $T|_{\text{supp}(\mu)}$ and $S|_{\text{supp}(\nu)}$ are topologically conjugate (see Definition \ref{def:conjugacy}) via the homeomorphism $\Theta_y$. 
    
    Returning to analyzing \eqref{eq:assumption}, we have by Lemma \ref{lemma:1} that 
    \begin{equation}\label{eq:sup_eq}
       \Psi_{(y,S)}^{(m+1)}(\text{supp}(\nu)) = \text{supp}\Big(\Psi_{(y,S)}^{(m+1)}\#\nu\Big) = \text{supp}\Big(\Psi_{(y,T)}^{(m+1)}\#\mu\Big) = \Psi_{(y,T)}^{(m+1)}(\text{supp}(\mu)). 
    \end{equation}
    Now, let $x\in \text{supp}(\nu)$ be fixed, and note that by the equality of sets \eqref{eq:sup_eq} and the definition of the time-delay map (see Definition \ref{def:time_delay}), there must exist some $z\in \text{supp}(\mu)$, such that
    \begin{equation}\label{eq:point_eq}    (\lefteqn{\underbrace{\phantom{y(x),y(S(x)),\dots, y(S^{m-1}(x))}}_{\Psi_{(y,S)}^{(m)}(x)}}y(x),\overbrace{y(S(x)),\dots, y(S^{m-1}(x)), y(S^m(x))}^{\Psi_{(y,S)}^{(m)}(S(x))}) = (\lefteqn{\underbrace{\phantom{y(x),y(S(x)),\dots, y(S^{m-1}(x))}}_{\Psi_{(y,T)}^{(m)}(z)}}y(z),\overbrace{y(T(z)),\dots, y(T^{m-1}(z)), y(T^m(z))}^{\Psi_{(y,T)}^{(m)}(T(z))})\,.
    \end{equation}
    By equating the first $m$ components, and then the last $m$ components, of the $m+1$ dimensional vectors appearing in \eqref{eq:point_eq}, we obtain the following two equalities:
    \begin{align}
        \Psi_{(y,S)}^{(m)}(x) &= \Psi_{(y,T)}^{(m)}(z) \label{eq:start},\\
        \Psi_{(y,S)}^{(m)}(S(x))&=\Psi_{(y,T)}^{(m)}(T(z))\label{eq:end}.
    \end{align}
    Since $x\in \text{supp}(\nu)$ and $z\in \text{supp}(\mu)$, we deduce from \eqref{eq:start} that $z = \Theta_y(x)$. Substituting this equality into~\eqref{eq:end} yields
    \begin{equation}\label{eq:next}
          \Psi_{(y,S)}^{(m)}(S(x))=\Psi_{(y,T)}^{(m)}(T(\Theta_y(x)))
    \end{equation}

    Moreover, since $T\# \mu = \mu$ and $S\# \nu = \nu$, it follows by Lemma \ref{lemma:1} that $T(\text{supp}(\mu)) = \text{supp}(\mu)$ and $S(\text{supp}(\nu)) = \text{supp}(\nu)$. Thus, $S(x)\in \text{supp}(\nu)$ and $T(\Theta_y(x))\in \text{supp}(\mu)$. This allows us to rearrange~\eqref{eq:next} to find 
    $$S(x) = (\Theta_y^{-1}\circ T \circ \Theta_y)(x).$$
    Since $x\in \text{supp}(\nu)$ was arbitrary, we have the general equality 
    \begin{equation}\label{eq:conjugate}
        S|_{\text{supp}(\nu)} = \Theta_y^{-1} \circ T|_{\text{supp}(\mu)} \circ \Theta_y,
    \end{equation}
   which completes the proof.
    \end{proof}

    Now that Theorem \ref{thm:1} has been established, we can easily deduce Corollary \ref{cor:1S}.
    \begin{proof}[Proof of Corollary \ref{cor:1S}] By Theorem \ref{thm:1}, we have that $\hat{\mu}_{(y,T)}^{(m+1)} = \hat{\nu}^{(m+1)}_{(y,S)}$ implies that $S|_{\text{supp}(\nu)} = \Theta_y^{-1} \circ T|_{\text{supp}(\mu)} \circ \Theta_y$ for a prevalent set $y\in \mathcal{Y}\subseteq C^1(U);$ see \eqref{eq:conjugate}. Now let $y\in \mathcal{Y}$ be fixed. Since $\Theta_y^{-1}:\text{supp}(\mu)\to \text{supp}(\nu)$ is a homeomorphism, we have by Proposition \ref{prop:conjugacy} that $\Theta_y^{-1}\# \mu$ is invariant and ergodic under $S|_{\text{supp}(\nu)}$. Since $\Theta_y^{-1}\#\mu = \nu$, which follows directly from the definition of the map $\Theta_y$ from~\eqref{eq:theta}, it holds that $\nu$ is $S$-ergodic. 
    \end{proof}
\subsubsection{Proof of Theorem \ref{thm:2}}\label{subsec:proofs2}
This section contains a proof of Theorem \ref{thm:2}. We begin by establishing several lemmas which are needed in our proof of the result. First, we will consider the case when two systems agree along an orbit, i.e., $S^k(x^*) = T^k(x^*)$ for all $k\in \mathbb{N}$. We will show that when $x^*\in B_{\mu,T}$ that the equality of $S$ and $T$ along the orbit initiated at $x^*$ implies that $S$ and $T$ agree everywhere on $\text{supp}(\mu)$. 
\begin{lemma}\label{lemma:7}
    Let $S,T:U\to U$ be continuous maps on an open set $U\subseteq \mathbb{R}^n$, let $\mu\in \mathcal{P}(U)$, and let $\textup{supp}(\mu)\subseteq U$ be compact. If for some $x^*\in B_{\mu,T}$ it holds that $S^k(x^*) = T^k(x^*)$ for all $k\in \mathbb{N}$, then $S|_{\textup{supp}(\mu)} = T|_{\textup{supp}(\mu)}$.
\end{lemma}
\begin{proof}
 Define the continuous function $\phi \in C(U)$ by setting $\phi(x):=\|S(x)-T(x)\|_2\geq 0$, for each $x\in U.$ 
Since $x^*\in B_{\mu,T}$, it follows from Definition \ref{def:basin} that 
\begin{align*}
    \int_{U}\|S(x) - T(x)\|_2d\mu(x) =\lim_{N\to \infty}\frac{1}{N}\sum_{k=0}^{N-1} \phi(T^k(x^*)).
\end{align*}
Moreover, since $T^k(x^*)=S^k(x^*)$ for all $k\in\mathbb{N}$, we find
 \begin{align*}
 \lim_{N\to \infty}\frac{1}{N}\sum_{k=0}^{N-1} \phi(T^k(x^*)) &=\lim_{N\to \infty}\frac{1}{N}\sum_{k=0}^{N-1} \| S(T^k(x^*))-T^{k+1}(x^*)\|_2\\& =\lim_{N\to \infty}\frac{1}{N}\sum_{k=0}^{N-1} \| S^{k+1}(x^*))-T^{k+1}(x^*)\|_2\\
 &= 0.
 \end{align*}
 Since $\|S(x) - T(x)\|_2 \geq 0$ for all $x\in U$, it must hold that $S(x) = T(x)$ for $\mu$-almost all $x\in U$.  
 
 We now define $C:=\{x\in U:S(x) = T(x)\}$ and note $C = \{x\in U: S(x) - T(x) = 0\} = (S-T)^{-1}(\{0\}).$ Since $S-T$ is continuous, it holds that $C$ is closed. Moreover, since $\mu(C) = 1$ it follows by Definition \ref{def:support} that $\text{supp}(\mu)\subseteq C$. Thus, $S|_{\textup{supp}(\mu)} = T|_{\textup{supp}(\mu)},$ as wanted.
\end{proof}
In our proof of Theorem \ref{thm:2}, we will require that $\Psi_{(y_i,T)}^{(m)}$ is injective on $\text{supp}(\mu)$ for each $1\leq i \leq m$, where $Y = (y_1,\dots, y_m)\in C^1(U,\mathbb{R}^m).$ The following result (Lemma \ref{lemma:augment}) is used to show that the set of all $Y\in C^1(U,\mathbb{R}^m)$ for which this property holds is prevalent.
\begin{lemma}\label{lemma:augment}
 Assume that $\mathcal{Y}\subseteq C^1(U,\mathbb{R})$ is prevalent. Then, the set of functions $\bm{\mathcal{Y}}:=\{(y_1,\dots, y_m): y_i \in \mathcal{Y}\}$ is prevalent in $C^1(U,\mathbb{R}^m).$   
\end{lemma}
\begin{proof}
We first show that $\bm{\mathcal{Y}}\subseteq C^1(U,\mathbb{R}^m)$ is Borel. Towards this, define the projection $$\bm{\pi}_i: C^1(U,\mathbb{R}^m)\to C^1(U,\mathbb{R}), \qquad \bm{\pi}_i((y_1,\dots, y_m)):= y_i\in C^1(U,\mathbb{R}),$$
for all $(y_1,\dots, y_m)\in C^1(U,\mathbb{R}^m)$ and each $1\leq i \leq m.$. The projection $\bm{\pi}_i$ is continuous, hence Borel measurable. Since $\mathcal{Y}$ is Borel, it holds that $\bm{\pi}_i^{-1}(\mathcal{Y})$ is also Borel for $1\leq i \leq m$. Then, we can write $\bm{\mathcal{Y}}=\bigcap_{i=1}^m \bm{\pi}_i^{-1}(\mathcal{Y})$, which verifies that $\bm{\mathcal{Y}}\subseteq C^1(U,\mathbb{R}^m)$ is Borel. 

Since $\mathcal{Y}$ is prevalent, there exists a $k$-dimensional probe space $E\subseteq C^1(U,\mathbb{R})$ admitting a basis $\{e_i:1\leq i \leq k\}$, such that for any $y\in C^1(U,\mathbb{R})$, it holds that $y+ \sum_{i=1}^k a_ie_i \in \mathcal{Y}$ for Lebesgue-almost every $(a_1,\dots, a_k)\in \mathbb{R}^k$. Next, we will define the augmented probe space $$\bm{E}:=\{(v,\dots, v): v\in E\}\subseteq C^1(U,\mathbb{R}^m),$$ and note that $\{(e_i,\dots, e_i):1\leq i \leq k\}$ constitutes a basis for $\bm{E}.$  Now, fix $(y_1,\dots, y_m)\in C^1(U,\mathbb{R}^m)$ and observe that for each $1\leq j \leq m$, there exists a full Lebesgue measure set $B_j\subseteq \mathbb{R}^k$, such that $y_j+\sum_{i=1}^k a_ie_i \in \mathcal{Y}$ for all $(a_1,\dots, a_k)\in B_j$. By construction, the set $B:=\bigcap_{j=1}^k B_j$ has full Lebesgue measure in $\mathbb{R}^k$, and for all $(a_1,\dots,a_k)\in B$ it then holds that 
$$(y_1,\dots,y_m)+ \sum_{i=1}^{k}a_i(e_i,\dots, e_i) = \Bigg(y_1 + \sum_{i=1}^k a_i e_i,\dots, y_m + \sum_{i=1}^k a_i e_i\Bigg) \in \bm{\mathcal{Y}}, $$
which completes the proof.

\end{proof}

In our proof of Theorem \ref{thm:2}, we will utilize the generalized Takens theorem (Theorem \ref{thm:takens}) to conclude that each $\Psi_{(y_i,T)}^{(m)}$ is injective on the compact set $\text{supp}(\mu)$, for each $1\leq i \leq m$. We would also like to use the generalized Whitney theorem (Theorem~\ref{thm:whitney}) to conclude that $Y=(y_1,\dots, y_m)$ is injective on $\text{supp}(\mu)$. However, in the generalized Whitney theorem it is assumed that $Y\in C^1(\mathbb{R}^n,\mathbb{R}^m)$, whereas in our statement of Theorem~\ref{thm:2} we have $Y\in C^1(U,\mathbb{R}^m)$, for an arbitrary open set $U \supseteq \text{supp}(\mu)$. The following lemma leverages the Whitney extension theorem (see \cite{whitney1992analytic}) to provide a reformulation of the generalized Whitney embedding theorem in this setting.
\begin{lemma}\label{lemma:extend}
    Let $A\subseteq U \subseteq \mathbb{R}^n,$ where $A$ is compact and $U$ is open, set $d:=\textnormal{boxdim}(A)$, and let $m > 2d$ be an integer. For almost every smooth map $F\in C^1(U,\mathbb{R}^m)$, it holds that $F$ is one-to-one on $A$ and an immersion on each compact subset of a smooth manifold contained in $A$. 
\end{lemma}
\begin{proof}
    By Theorem \ref{thm:whitney}, the result holds when $U = \mathbb{R}^n$, and by Remark \ref{rmk:linear} a suitable probe is given by space of linear maps between $\mathbb{R}^n$ and $\mathbb{R}^m$. Hereafter, we will denote this space by $\mathcal{L}(\mathbb{R}^n,\mathbb{R}^m)$. In the case when $U \neq \mathbb{R}^n$, we claim that the restricted space of linear maps $E:=\{ L|_U: L\in \mathcal{L}(\mathbb{R}^n,\mathbb{R}^m)\} $ is a suitable probe. We remark that if $W:=\{ L_1,\dots, L_{mn}\}$ is a basis for $\mathcal{L}(\mathbb{R}^n,\mathbb{R}^m)$, which is $nm$-dimensional, then $\hat{W}:=\{L_1|_U,\dots, L_{mn}|_U\}$ is a basis for $E$, which remains $nm$-dimensional. 
    
    To verify that $\hat{W}$ is a basis, let  $\hat{L}\in E$, and note that $\hat{L} = L|_U$ for some $L\in \mathcal{L}(\mathbb{R}^n,\mathbb{R}^m)$. Since $W$ is a basis for $\mathcal{L}(\mathbb{R}^n,\mathbb{R}^m)$, we can write $L = \sum_{i=1}^{nm} a_i L_i$ for some coefficients $a_i\in\mathbb{R}$, $1\leq i \leq nm$, and thus $\hat{L} = \sum_{i=1}^{nm}a_iL_i|_U.$ Therefore, $\hat{W}$ spans $E$. To see that the elements of $\hat{W}$ are linearly dependent, assume that $\sum_{i=1}^{nm}a_iL_i|_U \equiv 0 \in \mathbb{R}^m $ for some coefficients $a_i\in\mathbb{R}$ with $1\leq i \leq nm$. By linearity and the fact that $U$ is open, this implies $\sum_{i=1}^{nm}a_iL_i \equiv 0 \in \mathbb{R}^m $, and since $\{L_i\}_{i=1}^{nm}$ forms a basis, this implies $a_i = 0$ for each $1\leq i \leq nm$.

    To complete the proof, it remains to show that $E$ is a probe space. Towards this, let $F\in C^1(U,\mathbb{R}^m)$ be fixed, and note that by the Whitney Extension Theorem there exists $\tilde{F}\in C^1(\mathbb{R}^n,\mathbb{R}^m)$, such that $\tilde{F}|_A = F|_A$; see \cite{whitney1992analytic}. Then, by Theorem \ref{thm:whitney} and Remark \ref{rmk:linear}, it holds that $\tilde{F}+\sum_{i=1}^{nm}a_i L_i$ is injective on $A$ for Lebesgue almost all $(a_1,\dots, a_{nm})\in \mathbb{R}^{nm}$. Since $\tilde{F}|_A = F|_A$, it follows that $F+\sum_{i=1}^{nm}a_iL_{i}|_U$ is injective on $A$ for Lebesgue almost all $(a_1,\dots, a_{nm})\in \mathbb{R}^{nm}$, which completes the proof.
\end{proof}

We are now ready to present the proof of Theorem \ref{thm:2}. 
\begin{proof}[Proof of Theorem \ref{thm:2}]
By Theorem \ref{thm:1}, the set of $y \in C^1(U,\mathbb{R})$ such that the equality of measures $\hat{\mu}_{(y,S)}^{(m+1)}=\hat{\mu}_{(y,T)}^{(m+1)}$ implies the topological conjugacy of $T|_{\text{supp}(\mu)}$ and $S|_{\text{supp}(\mu)}$ is prevalent. We denote this set by $\mathcal{Y}$. Moreover, by Lemma~\ref{lemma:augment} it holds that the set $\bm{\mathcal{Y}}:=\{(y_1,\dots, y_m): y_i \in \mathcal{Y}\}$ is prevalent in $C^1(U,\mathbb{R}^m).$ By Lemma~\ref{lemma:extend}, it follows that the set $\bm{\mathcal{Z}}\subseteq C^1(U,\mathbb{R}^m)$ of smooth maps that are injective on $\text{supp}(\mu)$ is also prevalent. 
 Since the finite intersection of prevalent sets is prevalent (see Lemma \ref{lemma:prevalence}) it holds that $\bm{\mathcal{W}}:=\bm{\mathcal{Y}}\cap \bm{\mathcal{Z}}$ is a prevalent subset of $C^1(U,\mathbb{R}^m)$. 

We now fix an  element $Y = (y_1,\dots, y_m)\in \bm{\mathcal{W}}$. Note that by the conclusion of Theorem \ref{thm:1}, we have 
    \begin{equation}\label{eq:supps}
        S|_{\textup{supp}(\mu)} = (\Theta_{y_i}^{-1} \circ T|_{\text{supp}(\mu)}\circ \Theta_{y_i}),\qquad 1\leq i \leq m,
    \end{equation}
    where $\Theta_{y_i}:\text{supp}(\mu)\to \text{supp}(\mu)$ is defined in \eqref{eq:theta}. Evaluating \eqref{eq:supps} at $x^*\in B_{\mu,T}\cap \text{supp}(\mu)$ and composing on both sides then yields 
        \begin{equation*}\label{eq:supps2}
        S|_{\textup{supp}(\mu)}^k(x^*) = (\Theta_{y_i}^{-1} \circ T|_{\text{supp}(\mu)}^k\circ \Theta_{y_i})(x^*),\qquad k\in \mathbb{N},\qquad 1\leq i \leq m.
    \end{equation*}
    Then, since $S^k(x^*) = T^k(x^*)$ for $0\leq k \leq m-1$, it holds that $\Theta_{y_i}(x^*) = x^*$ for $1\leq i \leq m$, which is a consequence of the definition of the delay map~\eqref{eq:delay_map} and the construction of $\Theta_{y_i}$; see \eqref{eq:theta}.  
    Therefore,
    \begin{equation}\label{eq:supps3}
        S|_{\textup{supp}(\mu)}^k(x^*) = (\Theta_{y_i} ^{-1}\circ T|_{\text{supp}(\mu)}^k)(x^*),\qquad k\in \mathbb{N},\qquad 1\leq i \leq m.
    \end{equation}
    Using the definition of $\Theta_{y_i}$, we now rearrange \eqref{eq:supps3} to find that 
     \begin{equation}\label{eq:supps4}
        \Psi_{(y_i,S)}^{(m)}(S|_{\textup{supp}(\mu)}^k(x^*)) = \Psi_{(y_i,T)}^{(m)}(T|_{\textup{supp}(\mu)}^k(x^*)),\qquad k\in \mathbb{N},\qquad 1\leq i \leq m.
    \end{equation}
    Again, using the definition of the delay map~\eqref{eq:delay_map}, and equating the first components of the vectors in \eqref{eq:supps4} reveals that $y_i(S^k(x^*)) = y_i(T^k(x^*))$ for all $i=1,\ldots, m$ and any $k\in\mathbb{N}.$ Recall that $Y 
 = (y_1,\ldots,y_m)\in \bm{\mathcal{W}} = \bm{\mathcal{Y}} \cap \bm{\mathcal{Z}}$ is injective.  
 Thus, $S^k(x^*) = T^k(x^*)$ for all $k\in \mathbb{N}$. The conclusion then follows from Lemma~\ref{lemma:7}, which completes the proof.
\end{proof}
\subsubsection{Proof of Proposition \ref{prop:graph_fullS}}\label{subsec:proofs3}
Rather than using each observation function $y_i$ of $Y = (y_1,\dots, y_m)$ to construct a different delay-coordinate invariant measure $\hat{\mu}_{(y_i,T)}^{(m+1)},$ one can additionally study the measure $(Y,Y\circ T)\#\mu$ to uniquely identify the dynamics $T$ on $\text{supp}(\mu).$ This is the content of Proposition \ref{prop:graph_fullS}, which we prove here for completeness. The argument is analogous to our proof of Theorem \ref{thm:1}.

\begin{proof}[Proof of Proposition \ref{prop:graph_fullS}]
It follows by Lemma \ref{lemma:augment} that the set $$\mathcal{Y}:=\{Y\in C^1(U,\mathbb{R}^m): Y \text{ is injective on }{\text{supp}(\mu)}\}$$ is prevalent in $C^1(U,\mathbb{R}^m)$. Now, let $Y\in\mathcal{Y}$ be fixed and notice by Lemma \ref{lemma:1} that 
$$(Y,Y\circ T)(\text{supp}(\mu)) = \text{supp}((Y,Y\circ T)\#\mu) = \text{supp}((Y,Y\circ S)\#\mu) = (Y,Y\circ S)(\text{supp}(\mu)),$$
and thus 
\begin{equation}\label{eq:eq_final}
    \{(Y(x),Y(T(x))):x\in \text{supp}(\mu)\} = \{(Y(x),Y(S(x))):x\in \text{supp}(\mu)\}.
\end{equation}
Thus, for fixed $x\in \text{supp}(\mu)$ it follows by \eqref{eq:eq_final} that there is $z\in \text{supp}(\mu)$ such that 
$$(Y(x),Y(T(x))) = (Y(z),Y(T(z))).$$
Analyzing the first components we have $Y(x) = Y(z)$, and since $Y$ is injective on $\text{supp}(\mu)$ it holds that $x = z.$ Thus, $Y(T(x))=Y(S(x))$, which again  the injectivity of $Y$ on $\text{supp}(\mu)$ implies that $T(x) = S(x)$. Since $x\in \text{supp}(\mu)$ was chosen arbitrarily, we have $T|_{\text{supp}(\mu)} = S|_{\text{supp}(\mu)}$.
\end{proof}
\section{Numerical Experiments}\label{sec:full_exp}

\subsection{Code Availability}
Our code for the numerical experiments presented in the main text is publicly available on GitHub \cite{code}. Our GitHub also includes two tutorial notebooks which contain a step-by-step walk-through of the implementation of Algorithm 1 and Algorithm 2 presented in the main text. While Tutorial 1 focuses on the Kuramoto--Sivashinksy example presented in the main text and Tutorial 2 focuses on the Lorenz-63 example, the code we have developed is flexible and can be easily adapted to the user's application of choice.

\subsection{Discussion on Numerical Algorithms}
In the main text, we introduced Algorithm 1 and Algorithm 2 for performing data-driven parameter estimation using delay-coordinate invariant measures. Here we comment on a few key details required for the implementation of these algorithms, including the choice of embedding parameters, the choice of objective function, and gradient computations. 

Both Algorithms 1 and 2 require the construction of delay-coordinate invariant measures, which means that the embedding parameters must be estimated from data. While in our numerical experiments we have found that heuristic choices produce strong results, several approaches for automating and optimizing the selection of these embedding parameters exist.  For example, Cao's method can be used to select an appropriate embedding dimension \cite{cao1997practical}, while approaches based on the mutual information can be used to calibrate the delay parameter \cite{fraser1986independent}. 

We now clarify some quick notational details regarding the time-delay parameter. In Algorithms~1~and~2, we have assumed the available data is evenly sampled with $\Delta t_{\text{samp}} > 0$ and have used $\overline{\tau}\in \mathbb{N}$ to denote the discrete time-delay where $\tau = \overline{\tau}\Delta t_{\text{samp}}$. When modeling continuous-time dynamical systems,  $T_{\theta}$ should thus be regarded as the time-$\tau$ flow map of a vector field for the proper construction of the delay-coordinate invariant measures appearing in Algorithm~2. %

The choice of metric or divergence $\mathcal{D}:\mathcal{P}(\mathbb{R}^n)\times \mathcal{P}(\mathbb{R}^n)\to [0,\infty)$ over the probability space can significantly impact the performance of Algorithms 1 and 2. For high-dimensional applications, it is necessary that $\mathcal{D}$ can be directly evaluated on a point cloud. Some choices of $\mathcal{D}$ suitable for such high-dimensional applications include the Maximum Mean Discrepancy (MMD) and the Wasserstein distance, derived from the theory of optimal transport. Variants of the Wasserstein distance, such as the Sinkhorn divergence or the sliced Wasserstein distance, can also help accelerate computations. Several of these loss functions, including MMD and the Sinkhorn divergence, are implemented in the \texttt{Geomloss} library~\cite{feydy2019interpolating}, which is compatible with Pytorch's autograd engine and enables seamless gradient computations when using Algorithm 2. 

We conclude by noting that certain modifications to the objective functions presented in the algorithms from the main text can be useful in practice, as highlighted in the following examples. First, in Section~\ref{subsec:lorenz}, we discuss an alternative formulation of the training loss that is particularly advantageous when only finite-sample empirical distributions are available. Second, in Section~\ref{subsec:flow}, we demonstrate how information about the trajectory's initial condition, specifically, the first condition in Theorem~\ref{thm:2}, can be incorporated into the optimization procedure when aligning delay-coordinate invariant measures.

\subsection{Kuramoto--Sivashinsky Equation}

We now recall the experimental details for our parameter identification results of the Kuramoto--Sivashinsky (KS) equation presented in the main text. The equation is solved using an exponential time-differencing (ETD) scheme \cite{Koehler_Machine_Learning_and} on the spatio-temporal domain $[0,100]\times [0,10^4]$, subject to periodic spatial boundary conditions. We use an initial condition of $u_0(x) = \sin(\pi x/50)$ and discretization parameters $\Delta x = 0.5$ and $\Delta t = 0.1$. The dynamics are partially observed via $y(u(x,t)) = u(0,t)$ and the trajectory is slowly sampled with $\Delta t_{\text{samp}} = 3.0$. The partial observations are further contaminated with Gaussian measurement errors having mean zero and standard deviation $\sigma = 0.25$. To form the delay-coordinate invariant measure $\hat{\mu}_y$ from this partially observed noisy data, we use an embedding dimension of $m = 5$ and time-delay equal to the sampling interval, i.e., $\tau = 3.0$. The measure $\hat{\mu}_y\in \mathcal{P}(\mathbb{R}^5)$ is represented as an empirical distribution from the observed data and consists of approximately $3\cdot 10^3$ samples. 

For a given parameter $\theta\in [0.5,1.5]$, we generate a new solutions of the KS equation using the ETD scheme from which we extract the time-series $\{y_{\theta}(t_i)\}\subseteq \mathbb{R}$ and corresponding delay-coordinate invariant measure $\hat{\mu}_y(\theta).$ We then consider two methods of comparison for inferring the ground truth parameter $\theta = 1$ from the observed data: (1) a pointwise comparison between simulated and observed time-series, e.g., $\sum_i\|y(t_i)-y_{\theta}(t_i)\|_2$, and (2) a comparison between the simulated and observed delay-coordinate invariant measures, e.g., $\mathcal{D}(\hat{\mu}_y(\theta),\hat{\mu}_y)$, where in this experiment we select $\mathcal{D}$ as the sliced Wasserstein distance. 

We then use the Nelder--Mead optimizer to perform parameter identification using both pointwise and measure-based objectives. Over ten trials with the initial parameter guess uniformly sampled over $[0.5,1.5]$, the identification based on delay-coordinate invariant measures achieved an absolute error of $0.026\pm 0.019$, whereas the pointwise estimator incurred a larger error of $0.386\pm 0.140.$ Due to the chaotic behavior of the KS equation, the pointwise approach is unable to reliably detect the ground truth parameter, whereas the comparison based on delay-coordinate invariant measures is considerably more robust.

\subsection{Lorenz-63 System}\label{subsec:lorenz}
Our observed data for the Lorenz-63 system consists of a long noise-free trajectory generated via Euler's method. The data $\{u(t_i)\}\subseteq \mathbb{R}^3$ is obtained over the time interval $[0,2\cdot 10^3]$ using a time step of $\Delta t = 0.01$. The neural network $v_{\theta}:\mathbb{R}^3\to \mathbb{R}^3$ used to parameterize the unknown velocity is a multi-layer perceptron with hyperbolic tangent activation and layer structure $3\to 100 \to 100\to 100 \to 3$. To learn the dynamics, we consider two different loss functions based on the comparison of invariant measures. 

First, we consider $\mathcal{J}_{1}(\theta) = \mathcal{D}(T_{\theta}\#\mu, \mu) $, where  $\mu$ is the ground truth invariant measure and $T_{\theta}$ is the time-$\tau \Delta t$ flow map of the vector field $v_{\theta}$ with $\tau = 0.1$. When $\mathcal{J}_1$ is reduced to zero we have that $T_{\theta}\#\mu = \mu$, implying that $\mu$ is an invariant measure for $T_{\theta}.$  While the loss function $\mathcal{J}_1$ enforces equality of state-coordinate invariant measures, following Algorithm 2 in the main text we also consider $\mathcal{J}_2(\theta) = \mathcal{D}(T_{\theta}\#\mu, \mu)+ \mathcal{D}(\Psi_{(y,T_{\theta})}\#\mu, \hat{\mu}_{y})$, where $y:\mathbb{R}^3\to \mathbb{R}$ is the projection onto the first component, i.e., $y(x) = x\cdot e_1$, where $e_1\in \mathbb{R}^3$ is the first standard basis vector.  When $\mathcal{J}_2(\theta)$ is reduced to zero we have that $T_{\theta}\#\mu = \mu$, and thus $\Psi_{(y,T_{\theta})}\#\mu$ is the delay-coordinate invariant measure for $T_{\theta}$ which equals the observed delay-coordinate invariant measure $\hat{\mu}_y.$

In practice, we use a modification of the loss functions $\mathcal{J}_1$ and $\mathcal{J}_2$, i.e., 
\begin{equation}\label{eq:unbiased}
    \widetilde{\mathcal{J}}_1(\theta) = \mathcal{D}(T_{\theta}\#\mu,T\#\mu) ,\qquad \widetilde{\mathcal{J}}_2(\theta) = \mathcal{D}(T_{\theta}\#\mu,T\#\mu) + \mathcal{D}(\Psi_{(y,T_{\theta})}\#\mu,\Psi_{(y,T)}\#\mu)
\end{equation}
where $T$ is the time-$\tau \Delta t$ flow map of the ground truth Lorenz-63 system. Though the formulations in~\eqref{eq:unbiased} are equivalent to $\mathcal{J}_1$ and $\mathcal{J}_2$ in the infinite sample limit, they can have additional benefits when considering finite sample approximations to the underlying measures \cite{schiff2024dyslim}. We proceed to train the neural network models with the objectives \eqref{eq:unbiased} using the Adam optimizer with a learning rate of $10^{-3}$ for $10^4$ iterations. At each step of the optimization, the losses~\eqref{eq:unbiased} are approximated using $N = 500$ random samples from $\mu$ based on the long observed trajectory. In this experiment, $\mathcal{D}$ is chosen as the MMD based upon the so-called energy distance kernel; see \cite{feydy2019interpolating}. 

After training the neural network models 10 times with different random initializations, we evaluate performance by simulating long trajectories according to the learned vector fields. The reconstruction error is then quantified via the sliced Wasserstein distance between the simulated trajectory and the observed data samples. When only enforcing equality of invariant measures during model training, the errors are distributed over the interval $[10.07,1908]$ with a median of 393.58. When additionally enforcing the equality of the delay-coordinate invariant measures, the errors become significantly lower and are distributed between $[0.65, 1.59]$ with a median of $0.91.$

\subsection{Flow Past Cylinder}\label{subsec:flow}

The inference data for our cylinder flow example is generated using the D2Q9 Lattice Boltzman method \cite{Koehler_Machine_Learning_and}. Using a mesh spacing of $\Delta x = 0.1$ and a time-step of $\Delta t = 10^{-3}$, we simulate a solution with Re $= 70$ within the physical domain $\Omega = [0,15]\times [0,5]$, with the top and bottom of the domain subject to periodic boundary conditions. We consider 7 distinct probes placed in the resulting vortex street at locations $\{c_j\}_{j=1}^7\subseteq \Omega$ which measure the flow speed, subject to uncertainty in the location measurement. That is, every time the flow speed is recorded by a probe, the measurement is taken at a location randomly sampled from a normal distribution with mean $c_j$ and standard deviation $\sigma = 0.05.$ 

Our training training data consists of a sparse, noisy, and slowly sampled time-series from each probe $\{y_j(t_i)\}_{i=1}^{N}$ of length $N = 200$ and sampling interval $\Delta t _{\text{samp}} = 0.1$. We then regard $\{(y_1(t_i),\dots, y_7(t_i))\}_{i=1}^N \subseteq \mathbb{R}^7$ as the trajectory of an autonomous dynamical system, which we seek to model using a neural network parameterized velocity $v_{\theta} : \mathbb{R}^7\to \mathbb{R}^7$ with layer structure $7\to 100 \to 100 \to 100\to 7$. From each time-series $\{y_j(t_i)\}_{i=1}^N$, we construct a sample-based approximation to the delay-coordinate invariant measure $\hat{\mu}_{y_i}\in \mathcal{P}(\mathbb{R}^m)$ using the embedding dimension $m = 4$ and time-delay $\tau = 0.2.$ 

Following Algorithm 2, we then train a model for the velocity $v_{\theta}$, based upon the loss 
\begin{equation}\label{eq:objinit}
  \mathcal{J}(\theta) = \mathcal{D}(T_{\theta}\#\mu, \mu) + \sum_{j=1}^{7} \mathcal{D}(\Psi_{(y_j,T_{\theta})}\#\mu,\hat{\mu}_{y_j}) + \frac{1}{7m}\sum_{j=1}^7\sum_{k=0}^{m-1}\|y_j(t_k;\theta) - y_j(t_k)\|_2^2,
\end{equation}
where $T_{\theta}$ is the time-$\tau\Delta t_{\text{samp}}$ flow map of $v_{\theta}$, and again we use the energy distance MMD for $\mathcal{D}$.
While the first two terms of \eqref{eq:objinit} comprise the objective in Algorithm 2, the last term ensures that the initial evolution of the learned trajectory aligns with the initial evolution of the data collected from the flow sensors. Incorporating this information about the initial condition is crucial when one wants to make pointwisely accurate forecasts. It is also a necessary ingredient in Theorem \ref{thm:2} for proving uniqueness of the reconstructed dynamics from multiple delay-coordinate invariant measures. We lastly remark that, for this example, the pointwise matching in \eqref{eq:objinit} only involves matching the first $m=4$ samples of the evolving trajectory and represents a mild condition compared to pointwise matching of the entire dataset.

Similar to \eqref{eq:unbiased}, we use an analagous reformulation of \eqref{eq:objinit}, given that only sparse finite-sample approximations of the invariant measure $\mu$ and delay-coordinate invariant measures $\hat{\mu}_{y_j}$ are available. After training models for the velocity $v_{\theta}$ using both the delay measure objective \eqref{eq:objinit} and an objective based only on the invariant measure in the state coordinate, we find that the approach based on delay measures can accurately predict the evolution of the sensors past the training period on unseen testing data. In particular, the root mean-squared error (RMSE) for predicting a short time-series is $0.23\pm 0.15$ using the loss \eqref{eq:objinit}, while it is $2.85 \pm 0.007$ for an approach which only enforces equality of invariant measures in the state coordinates.

\end{document}